\documentclass{article}
\usepackage[utf8]{inputenc}
\usepackage{graphicx}
\usepackage{amsmath}
\usepackage{amsfonts}
\usepackage{lmodern}
\usepackage[T1]{fontenc}
\usepackage[english]{babel}
\usepackage{tikz-cd}
\usepackage{yfonts}
\usepackage{amssymb}
\usepackage{mathtools}
\usepackage{array}
\usepackage{wasysym}
\usepackage{amsthm}
\usepackage{url}
\usepackage{comment}
\theoremstyle{definition}
\newtheorem{definition}{Definition}[section]
\newtheorem{proposition}{Proposition}[definition]
 
\newtheorem{corollary}{Corollary}[definition] 
\newtheorem{lemma}{Lemma}[definition]
\newtheorem{oss}{Remark}[definition]
\newtheorem{es}{Example}[definition]

\usepackage{enumitem}
\newlist{nicenum}{enumerate}{1}
\setlist[nicenum]{%
  label=\textbf{\arabic*.},%
  left=1em,%
  itemsep=0.5ex,%
  topsep=1ex%
}

\title{On the integrability of generalized almost complex structures on $\mathbb{S}^{6}$.}
\author{Andrea Ricciarini}
\date{}
\begin{document}
\maketitle
\begin{abstract}
\hspace{-0.65cm} We study integrability of generalized almost complex structures on the six-dimensional sphere $\mathbb{S}^{6}$. Two notions of integrability are considered: integrability with respect to brackets determined by an affine connection $[\cdot ,\cdot ]_{\nabla}$ (in particular the Levi-Civita connection), and the Courant integrability for strong generalized almost complex structures. After recalling the necessary background on the generalized tangent bundle and on spherical combinations of the canonical generalized structures determined by an almost Hermitian triple $(J,g,\omega )$, we derive local coordinate criteria for $[\cdot ,\cdot ]_{\nabla }$-integrability of weak generalized structures. Applying these formulae to the nearly K\"ahler structure on $\mathbb{S}^{6}$ induced by the octonionic product, we prove that no nontrivial spherical combinations $\mathsf{J}=a\mathsf{J}_{1,J}+b\mathsf{J}_{g}+c\mathfrak{J}_{\omega }$ with smooth coefficients such that $a^{2}+b^{2}+c^{2}=1$ (except $\pm \mathsf{J}_{g}$) are integrable with respect to $[\cdot ,\cdot ]_{\nabla ^{LC}}$. We then turn to Courant integrability: we give local characterizations for Courant integrability of strong generalized almost complex structures, prove a gluing result for local Courant algebroids and $b$-field transforms, and use it to exhibit obstructions to constructing, via certain gluing procedures, smooth strong generalized complex structures on $\mathbb{S}^{6}$.
\end{abstract}
\begin{small}
\end{small} 
\section{Introduction}
Generalized complex geometry was introduced by Nigel Hitchin in \cite{H} and further developed by Marco Gualtieri in \cite{G1}, \cite{G2}. It blends features of complex and symplectic geometry and is naturally formulated on the generalized tangent bundle $E=TM\oplus T^{*}M$, endowed with its canonical pairing and the Courant bracket. This framework has proven effective for both structural results and the explicit construction of nontrivial examples: many classical notions admit reinterpretations in the generalized setting, and new phenomena arise (\cite{E}, \cite{G1}, \cite{N1}).
\\ In this paper we investigate integrability of generalized almost complex structures on the six-dimensional sphere $\mathbb{S}^{6}$. Two notions of integrability are relevant here: Courant integrability for strong generalized almost complex structures, and integrability with respect to brackets induced by an affine connection $[\cdot ,\cdot ]_{\nabla}$ (introduced in \cite{N1}). The six-dimensional sphere $\mathbb{S}^{6}$, endowed with its canonical nearly K\"ahler structure coming from octonionic product (\cite{E}), provides a natural and rich setting in which to examine these two integrability notions and to test their properties. The paper is organized as follows.
\\ In Section $2$ we recall the basic definitions from \cite{G1}, \cite{N1} and summarize the principal results that are used throughout the article. Following the presentation in \cite{E}, we introduce the nearly K\"ahler structure on the six-dimensional sphere induced by the octonionic product, and we define the spherical combinations, i.e. those weak generalized almost complex structures of the form 
$$\mathsf{J}=a\mathsf{J}_{1,J}+b\mathsf{J}_{g}+c\mathfrak{J}_{\omega},$$
where $\mathsf{J}_{1,J},\mathsf{J}_{g},\mathfrak{J}_{\omega}$ are the generalized almost complex structures associated to the nearly K\"ahler triple $(J,g,\omega )$, and $a,b,c \in C^{\infty}(\mathbb{S}^{6})$ satisfy $a^{2}+b^{2}+c^{2}=1$.
\\ In Section 3 we first characterize the integrability of a weak generalized almost complex structure with respect to the bracket $[\cdot,\cdot]_{\nabla}$ induced by an arbitrary affine connection, both intrinsically and in local coordinates, and we give a simple sufficient condition. These results are then applied to obtain the original result of the paper: the only spherical combinations on $\mathbb{S}^{6}$ that are integrable with respect to the bracket induced by the Levi-Civita connection are $\pm\mathsf{J}_{g}$.
\\ Finally, in Section 4 we analyze Courant integrability for strong generalized almost complex structures and give its characterization in local coordinates. We then study a gluing result for generalized tangent bundles and for generalized almost complex structures. Using this gluing construction, we show that certain gluing procedures, starting from locally flat models on the stereographic charts, cannot produce a smooth strong generalized almost complex structure on $\mathbb{S}^{6}$.
\section{Preliminaries}
Let $M$ be a smooth manifold, let $TM$ be its tangent bundle and $T^{*}M$ its cotangent bundle. In what follows we shall denote the smooth sections of $TM$ by $C^{\infty}(TM)$, of $T^{*}M$ by $C^{\infty}(T^{*}M)$.
\subsection{Geometry of the generalized tangent bundle}
\begin{definition}
\cite{G1,H} Let $M$ be a smooth manifold. The \emph{generalized tangent bundle} is the vector bundle over $M$
$$E:=TM \oplus T^{*}M.$$
Sections of $E$, $C^{\infty}(E)$, are written as $X+\xi $ with $X \in C^{\infty}(TM)$, $\xi \in C^{\infty}(T^{*}M)$.
\\ $E$ carries the natural symmetric pairing of neutral signature
$$\langle X+\xi ,Y+\eta \rangle :=\iota _{X}\eta +\iota _{Y}\xi .$$
The projection $\pi:E\rightarrow TM $ is the projection on the first summand. 
\\ The subbundles $TM$ and $T^{*}M$ are isotropic for $\langle \cdot ,\cdot \rangle $.
\end{definition}
\begin{definition}
\cite{G1,H} The \emph{Dorfman bracket} is the bilinear operation on $C^{\infty}(E)$, $[\cdot ,\cdot ]_{D}$, defined by:
$$[X+\xi ,Y+\eta ]_{D}:=[X,Y]+\mathcal{L}_{X}\eta -\iota _{Y}d\xi ,$$
where $[X,Y]$ is the Lie bracket and $\mathcal{L}$ the Lie derivative. 
\\ The Dorfman bracket satisfies the Jacobi rule:
$$[e_{1},[e_{2},e_{3}]_{D}]_{D}=[[e_{1},e_{2}]_{D},e_{3}]_{D}+[e_{2},[e_{1},e_{3}]_{D}]_{D}$$
$\forall e_{1},e_{2},e_{3} \in C^{\infty}(E)$. The Dorfman bracket is not skew-symmetric.
\end{definition}
\begin{definition}
\cite{G1,H} The \emph{Courant bracket}, $[\cdot ,\cdot ]_{C}$, is the skew-symmetrization of the Dorfman bracket:
$$[e_{1},e_{2}]_{C}:=\frac{1}{2}([e_{1},e_{2}]_{D}-[e_{2},e_{1}]_{D}).$$
It is skew and satisfies the Jacobi identity up to an exact term controlled by the pairing:
$$\sum _{cyc}[[e_{1},e_{2}]_{C},e_{3}]_{C}=\frac{1}{3}d\left(\sum _{cyc} \langle [e_{1},e_{2}]_{C},e_{3}\rangle \right) ,$$
where cyc means cyclic permutations.
\end{definition}
\begin{definition} 
$\cite{Bar,G1}$ Let $b \in \Omega^{2}(M)$ be a $2$-form. The \emph{$b$-transform} of $E$ is the $\langle \cdot ,\cdot \rangle $-orthogonal automorphism of $E$ defined as:
$$e^{b}(X+\xi )=X+\xi +\iota _{X}b$$
for all $X+\xi \in C^{\infty}(E)$. 
\\ If, in addition, $b$ is $d$-closed, then $e^{b}$ preserves the Courant bracket and is referred to as \emph{$b$-field transform}.
\end{definition}
\begin{definition} 
\cite{G1,H} Let $H \in \Omega ^{3}(M)$ be a given $3$-form. The \emph{H-twisted Dorfman bracket} on $C^{\infty}(E)$, $[\cdot ,\cdot ]_{D}^{H}$, is defined by:
$$[X+\xi ,Y+\eta ]_{D}^{H}=[X,Y]+\mathcal{L}_{X}\eta -\iota _{Y}d\xi +\iota _{X}\iota _{Y}H,$$
for $X+\xi ,Y+\eta \in C^{\infty}(E)$.
\\ The $H$-twisted Courant bracket, $[\cdot ,\cdot ]_{C}^{H}$, is the skew-symmetrization of the $H$-twisted Dorfman bracket:
$$[X+\xi ,Y+\eta ]_{C}^{H}=\frac{1}{2}([X+\xi ,Y+\eta ]_{D}^{H}-[Y+\eta ,X+\xi ]_{D}^{H})=$$
$$=[X,Y]+\mathcal{L}_{X}\eta -\mathcal{L}_{Y}\xi -\frac{1}{2}d(\iota _{X}\eta -\iota _{Y}\xi )+\iota _{X}\iota _{Y}H.$$
The twisted brackets preserve the canonical pairing $\langle \cdot ,\cdot \rangle $ and are equivariant under $b$-field transforms.
\end{definition}
\begin{definition}
\cite{N1} Given an affine connection $\nabla $ on $TM$, it is possible to define the \emph{connection-induced bracket} on $C^{\infty}(E)$, $[\cdot ,\cdot ]_{\nabla}$, as:
$$[X+\xi ,Y+\eta ]_{\nabla}=[X,Y]+\nabla _{X}\eta -\nabla _{Y}\xi .$$
This bracket is skew-symmetric and it satisfies the Jacobi identity if and only if the curvature of $\nabla $ vanishes.
\end{definition}

\subsection{Generalized complex structures}
\begin{definition}
\cite{G1,H} A \emph{generalized almost complex structure à la Hitchin} is an endomorphism $\mathfrak{J}:E \rightarrow E$ such that:
$$ \mathfrak{J}^{2}=-Id, \hspace{0.3cm} \langle \mathfrak{J}v,\mathfrak{J}w \rangle =\langle v,w \rangle \hspace{0.25cm} \forall v,w \in C^{\infty}(E).$$
Equivalently, the $+i$-eigenbundle $V_{i}\subset E\otimes \mathbb{C}$ is a maximal isotropic subbundle. The structure $\mathfrak{J}$ is integrable if and only if $V_{i}$ is closed under the Courant bracket (i.e. $V_{i}$ is involutive). 
\\ In the pure spinor formulation \cite{G2,H}, $E\otimes\mathbb{C}$ acts on differential forms via the Clifford action $$ (X+\xi)\cdot\varphi=\iota_{X}\varphi+\xi\wedge\varphi ,$$
and $V_{i}$ is the annihilator of a line subbundle $K\subset\wedge^{\bullet}T^{*}M\otimes\mathbb{C}$, the \emph{canonical line bundle} of $\mathfrak{J}$:
$$ V_{i}=\{\,v\in E\otimes\mathbb{C}\;:\; v\cdot\rho=0\,\},$$ where $\rho$ is any local nonvanishing section of $K$ (a \emph{pure spinor}). The condition $V_{i}\cap\overline{V_{i}}=0$ corresponds to $(\rho,\bar{\rho})\neq 0$, where $(\cdot,\cdot)$ denotes the Mukai pairing. Then $\mathcal{J}$ is Courant integrable if and only if, locally, $$d\rho=v\cdot\rho \qquad \text{for some } v\in C^{\infty}(E\otimes\mathbb{C}).$$
\\ A generalized almost complex structure integrable with respect to $[\cdot,\cdot]_{C}$ is called a \emph{generalized complex structure}.
\end{definition}
\begin{definition}
Let $\mathfrak{J}$ be an endomorphism of $E$ and let $[\cdot,\cdot]$ be a bilinear bracket on $C^{\infty}(E)$. We define the \emph{Nijenhuis operator} of $\mathfrak{J}$ with respect to $[\cdot,\cdot]$ as the $\mathbb{R}$-bilinear map:
$$N_{\mathfrak{J}}:C^{\infty}(E)\times C^{\infty}(E) \rightarrow C^{\infty}(E)$$
given by the formula:
$$N_{\mathfrak{J}}(v,w)=[\mathfrak{J}v,\mathfrak{J}w]-\mathfrak{J}[\mathfrak{J}v,w]-\mathfrak{J}[v,\mathfrak{J}w]+\mathfrak{J}^{2}[v,w]$$
for all $v,w \in C^{\infty}(E)$.
\end{definition}
\begin{oss}
$N_{\mathfrak{J}}$ defined above is an operator. Whether one calls it a tensor depends on extra properties of the bracket and of $\mathfrak{J}$.
\end{oss}
\begin{definition}
\cite{N1} In the most general sense, a \emph{generalized almost complex structure} on $M$ is an almost complex structure $\mathsf{J}$ on $E$, that is, an endomorphism $\mathsf{J}:E \rightarrow E$ such that $\mathsf{J}^{2}=-Id$.
\end{definition}
\begin{oss}
\cite{E} Generalized almost complex structures à la Hitchin are referred to in the literature as \emph{strong}, whereas those just introduced are called \emph{weak}.
\end{oss}
\begin{oss}
In general, one cannot speak of the Nijenhuis tensor with respect to the Courant bracket for weak generalized almost complex structures, since such structures are not a priori compatible with the natural pairing $\langle \cdot ,\cdot \rangle $.
\\ However, one can still speak of the Nijenhuis tensor with respect to the bracket $[\cdot,\cdot]_{\nabla}$. 
\\ For strong generalized almost complex structures, the Nijenhuis operator is indeed a tensor with respect to both the Courant bracket $[\cdot,\cdot]_{C}$ and the bracket $[\cdot,\cdot]_{\nabla}$, because in this case the structures are compatible with the pairing $\langle \cdot ,\cdot \rangle$.
\\ So a weak (or strong) generalized almost complex structure $\mathsf{J}$ is said to be \emph{$[\cdot,\cdot]_{\nabla }$-integrable} if its Nijenhuis tensor with respect to $[\cdot,\cdot]_{\nabla }$, $N_{\mathsf{J}}^{[\cdot,\cdot]_{\nabla }}$, vanishes.
\end{oss}
\begin{proposition}
\cite{Riccia} Let $\mathfrak{J}$ be a strong generalized almost complex structure. Then the following results hold:
\begin{nicenum}[label=\roman*.]
\item The Nijenhuis tensor of $\mathfrak{J}$ with respect  to $[\cdot,\cdot]_{C}$, $N_{\mathfrak{J}}$, vanishes if and only if the $+i$-eigenspace $V_{i}$ of $\mathfrak{J}$ is $[\cdot,\cdot]_{C}$-involutive.
\item The Nijenhuis tensor of $\mathfrak{J}$ with respect  to $[\cdot,\cdot]_{\nabla}$, $N_{\mathfrak{J}}^{[\cdot,\cdot]_{\nabla}}$, vanishes if and only if the $\pm i$-eigenspace $V_{\pm i}$ of $\mathfrak{J}$ are both $[\cdot,\cdot]_{\nabla }$-involutive (the same holds for weak structures).
\end{nicenum}
\end{proposition}
\begin{proposition}
\cite{N1} Let $\mathsf{J}=\left[ \begin{matrix} H && \alpha \\ \beta && K \end{matrix}\right]$ be an endomorphism of $E$ where \\ $H:TM\rightarrow TM$, $\alpha :T^*M\rightarrow TM$, $\beta :TM\rightarrow T^*M$, $K:T^*M\rightarrow T^{*}M$. Then $\mathsf{J}$ is a generalized almost complex structure if and only if the following conditions hold: 
$$\begin{cases} \beta \alpha =-(I+K^2) \\ \alpha \beta =-(I+H^2) \\ H\alpha =-\alpha K \\ \beta H=-K\beta \end{cases}.$$ 
Moreover $\mathsf{J}$ is invariant with respect to the scalar product $\langle \cdot ,\cdot \rangle $, i.e. it is strong, if and only if, in addition to the above conditions, the following also hold:
$$\begin{cases} K=-H^*\\ \alpha =-\alpha ^* \\ \beta =-\beta ^* \end{cases}$$
where $H^*:T^*M\rightarrow T^*M$ is the dual of $H$ defined by $H^*(\xi )(X)=\xi (H(X))$. The condition $\alpha =-\alpha ^*$ means that $\alpha (\xi )(\eta )=-\alpha (\eta )(\xi )$ for all $\xi ,\eta \in T^*M$ while $\beta = -\beta ^*$ means that $\beta (X)(Y)=-\beta (Y)(X)$ for all $X,Y \in C^{\infty}(TM)$. 
\end{proposition}
\begin{es}
Let $(M,g)$ be a Riemannian manifold. Then
$$\mathsf{J}_{g}=\left(\begin{matrix}
0 && -g^{-1}\\
g && 0 \\ 
\end{matrix}\right)$$
is a weak generalized almost complex structure on $M$.
\end{es}
\begin{es}
Let $(M,\omega)$ be an almost symplectic manifold. Then
$$\mathfrak{J}_{\omega}=\left(\begin{matrix}
0 && -\omega ^{-1}\\
\omega && 0 \\ 
\end{matrix}\right)$$
is a strong generalized almost complex structure on $M$.
\end{es}
\begin{es}
Let $(M,J)$ be an almost complex manifold. Then
$$\mathsf{J}_{1,J}=\left(\begin{matrix}
J && 0\\
0 && J^* \\ 
\end{matrix}\right)$$
is a weak generalized almost complex structure on $M$, while
$$\mathfrak{J}_{-1,J}=\left(\begin{matrix}
J && 0\\
0 && -J^* \\ 
\end{matrix}\right)$$
is a strong generalized almost complex structure.
\end{es}
\begin{proposition}
\cite{E} Let $(M,J,g)$ be an almost Hermitian manifold. Then, a linear combination $a\mathsf{J}_{\lambda ,J}+b\mathsf{J}_{g}+c\mathfrak{J}_{\omega}$ with $a,b,c \in C^{\infty}(M)$ is a weak generalized almost complex structure if and only if $\lambda =1$ and $a^{2}+b^{2}+c^{2}=1$. The structure is strong if and only if $\lambda =-1$ and $a\equiv \pm 1$, or $c\equiv \pm 1$.
\\ We will call such a $C^{\infty}(M)$-linear combination a \emph{spherical combination} of \\ $(\mathsf{J}_{1,J}, \mathsf{J}_{g},\mathfrak{J}_{\omega})$.
\end{proposition}

\subsection{Nearly K\"ahler structure on $\mathbb{S}^{6}$}
We now introduce the nearly K\"ahler structure on the six-dimensional sphere $(\mathbb{S}^{6},g,J)$ arising from the pure octonions product. For this purpose, we proceed as in \cite{E}.
\\ This will be useful, since in the next section we shall study the $\nabla^{LC}$-integrability of the spherical combination induced by this structure (here $\nabla^{LC}$ denotes the Levi–Civita connection).
\\ We will work in spherical coordinates for $\mathbb{S}^{6} \subset \mathbb{R}^{7}$. In particular, if we take angular coordinates $u^{1},\dots ,u^{6}$ such that $u^{1}, \dots ,u^{5} \in (0, \pi )$ and $u^{6} \in (0, 2\pi )$, the coordinates of $\mathbb{S}^{6}$ are:
$$ \begin{cases}
x^{1}=cos(u^{1}) \\ 
x^{2}=sin(u^{1})cos(u^{2})\\
\vdots \\ 
x^{6}=sin(u^{1})sin(u^{2})sin(u^{3})sin(u^{4})sin(u^{5})cos(u^{6})\\
x^{7}=sin(u^{1})sin(u^{2})sin(u^{3})sin(u^{4})sin(u^{5})sin(u^{6})
\end{cases}.$$
As shown in \cite{E}, by regarding $\mathbb{S}^{6}$ as an embedded submanifold of $\mathbb{R}^{7}$, one obtains the following expression for the metric $g$ in spherical coordinates:
$$g_{ij}=\begin{cases}
1  & \text{if } i=j=1 \\ 
sin^{2}(u^{1})\dots sin^{2}(u^{i-1}) & \text{if } i=j\neq 1\\ 
0 & \text{if } i\neq j.
\end{cases}$$
The local representation of the inverse metric follows directly from the previous equation:
$$g^{ij}=\begin{cases}
1  & \text{if } i=j=1 \\ 
\frac{1}{sin^{2}(u^{1})\dots sin^{2}(u^{i-1})} & \text{if } i=j\neq 1\\ 
0 & \text{if } i\neq j.
\end{cases}$$
The almost complex structure $J$ on $\mathbb{S}^{6}$ is induced by the multiplication of pure octonions in $\mathbb{R}^{7}$. Denoting this product by $\times$, the endomorphism $J_{p}:T_{p}\mathbb{S}^{6}\rightarrow T_{p}\mathbb{S}^{6}$ is defined at each point $p \in \mathbb{S}^{6}$ by
$$J_{p}(v)=p\times v$$
for every $v \in T_{p}\mathbb{S}^{6}$. An explicit description of the octonionic multiplication may be found, for instance, in \cite{Chen}.
As the full coordinate expression of $J$ is prohibitively long, we only record a few components $J^{i}_{j}$ computed in \cite{E}:
\begin{equation} \label{eq:1indicebasso}
J\partial_{1}=\frac{cos(u^{3})}{sin(u^{1})}\partial_{2}-\frac{cos(u^{2})sin(u^{3})}{sin(u^{1})sin(u^{2})}\partial_{3}+\frac{cos(u^{5})}{sin(u^{1})}\partial_{4}-\frac{cos(u^{4})sin(u^{5})}{sin(u^{1})sin(u^{4})}\partial_{5}-\frac{1}{sin(u^{1})}\partial_{6}.
\end{equation}
Therefore, as $J\partial_{1}=J_{1}^{k}\partial_{k}$ we have the expressions for $J_{1}^{k}$. Furthermore, the components $J_{k}^{1}$ are given by  $$J_{k}^{1}=-g_{kk}J_{1}^{k}.$$
\\ For any $X,Y \in C^{\infty}(T\mathbb{S}^{6})$, one checks that $g(JX,JY)=g(X,Y)$. Setting $\omega (X,Y)=g(JX,Y)$, Proposition 2.9.3 guarantees that every spherical combination of the generalized structures $(\mathsf{J}_{1,J},\mathsf{J}_{g},\mathfrak{J}_{\omega})$ is a weak generalized almost complex structure.

\section{$[\cdot ,\cdot ]_{\nabla}$-integrability}
\subsection{Integrability}
In this section we derive the conditions for $[\cdot ,\cdot ]_{\nabla}$-integrability of a weak generalized almost complex structure $\mathsf{J}$. We initially adopt the framework of $\cite{N1}$ and subsequently develop the resulting equations in local coordinates.
\begin{proposition}
Let $\nabla$ be an affine connection on $M$ and let $\mathsf{J}=\left( \begin{matrix} H && \alpha \\ \beta && K \end{matrix} \right) $ be a weak generalized almost complex structure on $M$. Let $N^{\nabla}(\mathsf{J})$ be the Nijenhuis tensor of $\mathsf{J}$ with respect to $[\cdot ,\cdot ]_{\nabla}$.
Then $\mathsf{J}$ is $[\cdot ,\cdot ]_{\nabla}$-integrable if and only if the following conditions hold for all $X,Y \in C^{\infty}(TM)$ and for all $\xi , \eta \in C^{\infty}(T^{*}M)$:

\begin{equation}
\begin{aligned}
N_{\mathsf{J}}^{[\cdot ,\cdot ]_{\nabla}}(X,Y)\big|_{C^{\infty}(TM)}
&=(\nabla_{HX}H)Y-(\nabla_{HY}H)X -H((\nabla_{X}H)Y +\\
&\quad -(\nabla_{Y}H)X) -\alpha ((\nabla_{X}\beta) Y-(\nabla_{Y}\beta ) X) +\\ 
&\quad -T^{\nabla}(HX,HY)+T^{\nabla}(X,Y)+H(T^{\nabla}(HX,Y) +\\ 
&\quad +T^{\nabla}(X,HY)) =0,
\end{aligned}
\end{equation}

\begin{equation} 
\begin{aligned}
N_{\mathsf{J}}^{[\cdot ,\cdot ]_{\nabla}}(X,Y)\big|_{C^{\infty}(T^{*}M)}
&=(\nabla_{HX}\beta )Y-(\nabla_{HY}\beta )X-\beta((\nabla _{X}H)Y +\\ 
&\quad -(\nabla _{Y}H)X)-K((\nabla _{X}\beta) (Y)-(\nabla_{Y} \beta )(X))+ \\
&\quad +\beta (T^{\nabla} (HX,Y)+
T^{\nabla} (X,HY))=0,
\end{aligned}
\end{equation}

\begin{equation} 
N_{\mathsf{J}}^{[\cdot ,\cdot ]_{\nabla}}(\xi ,\eta )\big|_{C^{\infty}(TM)}
=(\nabla _{\alpha \xi}\alpha )\eta -(\nabla _{\alpha \eta }\alpha )\xi -T^{\nabla}(\alpha \xi ,\alpha \eta )=0,
\end{equation}

\begin{equation} 
N_{\mathsf{J}}^{[\cdot ,\cdot ]_{\nabla}}(\xi ,\eta )\big|_{C^{\infty}(T^{*}M)}=(\nabla _{\alpha \xi}K)\eta -(\nabla _{\alpha \eta }K)\xi =0,
\end{equation}

\begin{equation} 
\begin{aligned}
N_{\mathsf{J}}^{[\cdot ,\cdot ]_{\nabla}}(X,\eta )\big|_{C^{\infty}(TM)}
&=(\nabla _{HX}\alpha )\eta -(\nabla _{\alpha \eta }H)X+ \\ 
&\quad -H((\nabla _{X}\alpha )\eta )-\alpha ((\nabla _{X}K) \eta )+ \\ 
&\quad -T^{\nabla }(HX,\alpha \eta )+H^{T^{\nabla }}(X, \alpha \eta )=0,
\end{aligned}
\end{equation}
\begin{equation} 
\begin{aligned}
N_{\mathsf{J}}^{[\cdot ,\cdot ]_{\nabla}}(X,\eta )\big|_{C^{\infty}(T^{*}M)}
&=(\nabla _{HX}K)\eta -(\nabla _{\alpha \eta }\beta )X-\beta ((\nabla _{X}\alpha )\eta )+ \\ 
&\quad -K((\nabla _{X}K)\eta )+ \beta (T^{\nabla}(X, \alpha \eta ))=0.
\end{aligned}
\end{equation}
where $T^{\nabla}(X,Y)=\nabla _{X}Y-\nabla _{Y}X-[X,Y]$ is the torsion of the affine connection $\nabla$.
\end{proposition}
\begin{proof}
Recall that a weak generalized almost complex structure is said to be $[\cdot ,\cdot ]_{\nabla } $-integrable if and only if its Nijenhuis tensor vanishes identically, i.e.
$$N_{\mathsf{J}}^{[\cdot ,\cdot ]_{\nabla}}(X+\xi ,Y+\eta )=0 \hspace{0.4cm} \forall (X+\xi ) ,(Y+\eta ) \in C^{\infty}(E).$$
Carrying out the computations using Definition $2.8$ with $\mathsf{J}$ and the bracket $[\cdot ,\cdot ]_{\nabla }$ yields the six conditions stated above. For example, we now compute condition $(4)$: 
$$N_{\mathsf{J}}^{[\cdot ,\cdot ]_{\nabla}}(\xi ,\eta )\big|_{C^{\infty}(TM)}=([\mathsf{J}\xi,\mathsf{J}\eta]_{\nabla}- \mathsf{J}[\mathsf{J}\xi ,\eta ]_{\nabla}-\mathsf{J}[\xi ,\mathsf{J}\eta ]_{\nabla}-[\xi ,\eta ]_{\nabla})\big|_{C^{\infty}(TM)}=$$
$$=([\alpha \xi +K\xi ,\alpha \eta +K \eta]_{\nabla}- \mathsf{J}[\alpha \xi +K\xi ,\eta]_{\nabla}-\mathsf{J}[\xi , \alpha \eta +K\eta]_{\nabla})\big| _{C^{\infty}(TM)}=$$
$$=([\alpha \xi ,\alpha \eta ]+\nabla _{\alpha \xi }K\eta -\nabla _{\alpha \eta}K\xi -\mathsf{J}\nabla _{\alpha \xi }\eta + \mathsf{J}\nabla _{\alpha \eta }\xi )\big| _{C^{\infty}(TM)}=$$
$$=(\nabla _{\alpha \xi }\alpha )(\eta)-(\nabla _{\alpha \eta }\alpha)(\xi )-T^{\nabla}(\alpha \xi ,\alpha \eta ).$$
\end{proof}

\begin{lemma}
Let $\nabla$ be an affine connection on $M$ and let $\mathsf{J}=\left( \begin{matrix} H && \alpha \\ \beta && K \end{matrix} \right) $ be a weak generalized almost complex structure on $M$. Let us take local coordinates $(U,(x^{1},\dots ,x^{n}))$ such that:
$$H\partial _{i}=H^{j}_{i}\partial _{j}, \hspace{0.4cm} \alpha dx^{i}=\alpha ^{ij}\partial_{j}, \hspace{0.4cm}
\beta \partial _{i}=\beta _{ij} dx^{j}, \hspace{0.4cm}
K dx^{i}=K_{j}^{i}dx^{j}.$$
Then $\mathsf{J}$ is $[\cdot ,\cdot ]_{\nabla}$-integrable if and only if the following conditions are met for all $i,j,k \in \{1, \dots , dim_{\mathbb{R}}M\}$:

\begin{equation} \label{eq:1nabla}
\begin{aligned}
(N_{\mathsf{J}}^{[\cdot ,\cdot]_{\nabla}}(\partial _{i},\partial _{j})\big|_{C^{\infty}(TM)})^{k}
&=(H_{i}^{p}(\partial _{p}H_{j}^{k}+\Gamma _{ps}^{k}H_{j}^{s}-\Gamma_{pj}^{s}H_{s}^{k})-H_{j}^{p}(\partial _{p}H_{i}^{k}+ \\ 
&\quad +\Gamma _{ps}^{k}H_{i}^{s}- \Gamma_{pi}^{s}H_{s}^{k}) -H_{k}^{l}(\partial _{i}H_{j}^{l}+\Gamma_{is}^{l}H_{j}^{s}+ \\ 
&\quad - \Gamma _{ij}^{s}H_{s}^{l}-\partial_{j}H_{i}^{l} -\Gamma_{js}^{l}H_{i}^{s} +\Gamma_{ji}^{s}H_{s}^{l}) -\alpha ^{lk}(\partial _{i}\beta _{jl}+ \\ 
&\quad -\Gamma _{ij}^{q}\beta _{ql}-\Gamma _{il}^{q}\beta _{jq}-\partial _{j}\beta _{il}+\Gamma _{ji}^{q}\beta _{ql}+\Gamma _{jl}^{q}\beta _{iq})+ \\ 
&\quad -H_{i}^{l}H_{j}^{p}(\Gamma _{lp}^{k}-\Gamma_{pl}^{k})+(\Gamma_{ij}^{k}-\Gamma _{ji}^{k})+ H_{m}^{k}(H_{i}^{p}(\Gamma_{pj}^{m}+\\ 
&\quad -\Gamma _{jp}^{m})+H_{j}^{l}(\Gamma _{il}^{m}-\Gamma _{li}^{m})))\partial _{k}=0,
\end{aligned}
\end{equation}

\begin{equation} \label{eq:2nabla}
\begin{aligned}
(N_{\mathsf{J}}^{[\cdot ,\cdot ]_{\nabla}}(\partial _{i},\partial _{j})\big|_{C^{\infty}(T^{*}M)})^{k}
&=(H_{i}^{p}(\partial _{p}\beta _{jk}-\Gamma _{pj}^{q}\beta _{qk}-\Gamma_{pk}^{q}\beta_{jq})-H_{j}^{p}(\partial_{p}\beta _{ik}+ \\ 
&\quad -\Gamma_{pi}^{q}\beta _{qk} -\Gamma _{pk}^{q}\beta _{iq}) - \beta _{lk}(\partial _{i}H_{j}^{l}+\Gamma _{is}^{l}H_{j}^{s}+ \\ 
&\quad -\Gamma _{ij}^{s}H_{s}^{l}-\partial _{j}H_{i}^{l}-\Gamma _{js}^{l}H_{i}^{s}+\Gamma _{ji}^{s}H_{s}^{l})-K_{k}^{l}(\partial _{i}\beta _{jl}+\\ 
&\quad -\Gamma _{ij}^{q}\beta _{ql}-\Gamma _{il}^{q}\beta _{jq}-\partial _{j}\beta _{il}+\Gamma _{ji}^{q}\beta _{ql}+\Gamma _{jl}^{q}\beta _{iq}) + \\ 
&\quad + \beta _{mk}(H_{i}^{p}(\Gamma _{pj}^{m}-\Gamma _{jp}^{m})+H_{j}^{l}(\Gamma_{il}^{m}-\Gamma _{li}^{m})))dx^{k}=0,
\end{aligned}
\end{equation}

\begin{equation} \label{eq:3nabla}
\begin{aligned}
(N_{\mathsf{J}}^{[\cdot ,\cdot ]_{\nabla}}(dx^{i},dx^{j})\big|_{C^{\infty}(TM)})^{k}
&=(\alpha ^{ip} (\partial _{p}\alpha ^{jk}+\Gamma _{ps}^{j}\alpha ^{sk}+ \Gamma _{ps}^{k}\alpha ^{js})-\alpha ^{jp}(\partial _{p}\alpha^{ik}+ \\ 
&\quad +\Gamma _{ps}^{i}\alpha ^{sk}+\Gamma _{ps}^{k}\alpha ^{is})-\alpha ^{il}\alpha ^{jm}(\Gamma _{lm}^{k}-\Gamma _{ml}^{k}))\partial _{k}=0,
\end{aligned}
\end{equation}

\begin{equation} \label{eq:4nabla}
\begin{aligned}
(N_{\mathsf{J}}^{[\cdot ,\cdot ]_{\nabla}}(dx^{i},dx^{j})\big|_{C^{\infty}(T^{*}M)})^{k}
&=(\alpha ^{ip}(\partial _{p}K _{j}^{k}+\Gamma _{pq}^{k}K_{j}^{q}-\Gamma _{pj}^{q}K_{q}^{k})-\alpha ^{jp}(\partial _{p}K_{i}^{k}+\\ 
&\quad +\Gamma _{pq}^{k}K_{i}^{q}-\Gamma _{pi}^{q}K_{q}^{k}))dx^{k}=0,
\end{aligned}
\end{equation}

\begin{equation} \label{eq:5nabla}
\begin{aligned}
(N_{\mathsf{J}}^{[\cdot ,\cdot ]_{\nabla}}(\partial _{i},dx^{j})\big|_{C^{\infty}(TM)})^{k}
&=(H_{i}^{p}(\partial _{p}\alpha ^{jk}+\Gamma _{ps}^{j}\alpha ^{sk}+\Gamma _{ps}^{k}\alpha ^{js})-\alpha ^{jp}(\partial _{p}H_{i}^{k}+ \\ 
&\quad +\Gamma _{ps}^{k}H_{i}^{s}-\Gamma _{pi}^{s}H_{s}^{k})-H_{l}^{k}(\partial _{i}\alpha ^{jl}+\Gamma _{is}^{j}\alpha ^{sl}+\\ 
&\quad +\Gamma _{is}^{l}\alpha ^{js}) -\alpha ^{lk}(\partial _{i}K_{j}^{l}+ \Gamma _{iq}^{l}K_{j}^{q}-\Gamma _{ij}^{q}K_{q}^{l})+\\ 
&\quad -H_{i}^{l}\alpha ^{jp}(\Gamma_{lp}^{k}-\Gamma _{pl}^{k})+H_{m}^{k}(\alpha ^{jp}(\Gamma_{ip}^{m}-\Gamma _{pi}^{m})))\partial _{k}=0,
\end{aligned}
\end{equation}

\begin{equation} \label{eq:6nabla}
\begin{aligned}
(N_{\mathsf{J}}^{[\cdot ,\cdot ]_{\nabla}}(\partial _{i},dx^{j})\big|_{C^{\infty}(T^{*}M)})^{k}
&=(H_{i}^{p}(\partial _{p}K_{j}^{k}+\Gamma _{pq}^{k}K_{j}^{q}-\Gamma _{pj}^{q}K_{q}^{k})-\alpha ^{jp}(\partial _{p}\beta _{ik}+ \\ 
&\quad -\Gamma _{pi}^{q}\beta _{qk}-\Gamma _{pk}^{q}\beta _{iq})-\beta _{lk}(\partial _{i}\alpha ^{jl}+\Gamma_{is}^{j}\alpha ^{sl}+ \\ 
&\quad +\Gamma _{is}^{l}\alpha ^{js})-K_{k}^{l}(\partial _{i}K_{j}^{l}+\Gamma _{iq}^{l}K_{j}^{q}-\Gamma _{ij}^{q}K_{q}^{l})+ \\ 
&\quad +\beta _{lk}(\alpha ^{jm}(\Gamma _{im}^{l}-\Gamma _{mi}^{l})))dx^{k}=0.
\end{aligned}
\end{equation}
\end{lemma}
\begin{proof}
Observe that, in local coordinates, the following equations hold:
$$((\nabla _{\partial _{p}}H)(\partial _{j}))^{k}=(\nabla _{p}H)_{j}^{k}=\partial _{p}H_{j}^{k}+\Gamma _{ps}^{k}H_{j}^{s}-\Gamma _{pj}^{s}H_{s}^{k},$$
$$((\nabla _{\partial _{p}}\alpha)(dx^{j}))^{k}=(\nabla _{p}\alpha)^{jk}=\partial _{p}\alpha ^{jk}+\Gamma _{ps}^{j}\alpha ^{sk}+\Gamma _{ps}^{k}\alpha ^{js},$$
$$((\nabla _{\partial _{p}}\beta)(\partial _{j}))^{k}=(\nabla _{p}\beta )_{jk}=\partial _{p}\beta _{jk}-\Gamma _{pj}^{q}\beta _{qk}-\Gamma _{pk}^{q}\beta _{jq},$$
$$(\nabla _{\partial _{p}}K)(dx^{j}))^{k}=(\nabla _{p}K)_{j}^{k}=\partial _{p}K_{j}^{k}+\Gamma _{pq}^{k}K_{j}^{q}-\Gamma _{pj}^{q}K_{q}^{k}.$$
Substituting these equations into the conditions of the preceding proposition yields the desires conclusion.
\end{proof}
\begin{corollary}
Let $\nabla$ be an affine connection on $M$ and let $\mathsf{J}=\left( \begin{matrix} H && \alpha \\ \beta && K \end{matrix} \right) $ be a weak generalized almost complex structure on $M$. A sufficient condition for the $[\cdot ,\cdot ]_{\nabla} $-integrability of $\mathsf{J}$ is that the connection $\nabla $ be symmetric, $\nabla H=$ \\ $=\nabla \alpha =0$, and that $\alpha $ be invertible.
\end{corollary}
\begin{proof}
If $\alpha $ is invertible and $\nabla \alpha =\nabla H=0$, then using Proposition 2.9.2 we have $\nabla K=\nabla \beta =0$. Therefore, all the equations in the preceding proposition are satisfied.
\end{proof}

\subsection{$[\cdot ,\cdot ]_{\nabla ^{LC}}$-integrability of spherical combinations}
We now attempt to use the equations just derived to study the $[\cdot ,\cdot ]_{\nabla ^{LC}}$-integrability of certain spherical combinations on $\mathbb{S}^{6}$, obtained from the strictly nearly K\"ahler structure $(\mathbb{S}^{6},J,g)$ induced by the octonions product, where $\nabla ^{LC}$ denotes the Levi-Civita connection associated with $g$.
\begin{oss}\label{oss:NK}
It is useful at this point to collect some properties of the nearly K\"ahler structure $(\mathbb{S}^{6},J,g)$, both in tensorial form and in normal coordinates. Let $\nabla^{LC}$ be the Levi-Civita connection and define $$T(X,Y,Z):=g\big((\nabla^{LC}_{X}J)Y,\,Z\big), \qquad X,Y,Z\in C^{\infty}(T\mathbb{S}^{6}).$$ Clearly $T\equiv 0$ if and only if $\nabla^{LC} J=0$.
\begin{enumerate}
\item[(i)] \emph{$T$ is a $3$-form.} Differentiating $g(JY,Z)=-g(Y,JZ)$ and using $\nabla^{LC} g=0$ we get $g((\nabla^{LC}_{X}J)Y,Z)=-g(Y,(\nabla^{LC}_{X}J)Z)$, so $T$ is skew in $(Y,Z)$. The nearly K\"ahler condition $(\nabla^{LC}_{X}J)X=0$ gives, by polarization, $$(\nabla^{LC}_{X}J)Y+(\nabla^{LC}_{Y}J)X=0,$$ so $T$ is skew in $(X,Y)$. Since the transpositions $(12)$ and $(23)$ generate $S_{3}$, $T$ is totally skew-symmetric.
\item[(ii)] Differentiating $J^{2}=-\mathrm{Id}$ we obtain $(\nabla^{LC}_{X}J)J+J(\nabla^{LC}_{X}J)=0$, whence $$T(X,JY,Z)=-g\big(J(\nabla^{LC}_{X}J)Y,Z\big)=g\big((\nabla^{LC}_{X}J)Y,JZ\big)=T(X,Y,JZ).$$
\item[(iii)] Combining (i) and (ii): $$T(JX,Y,Z)\overset{\mathrm{(i)}}{=}-T(Y,JX,Z)\overset{\mathrm{(ii)}}{=}-T(Y,X,JZ) \overset{\mathrm{(i)}}{=}T(X,Y,JZ)\overset{\mathrm{(ii)}}{=}T(X,JY,Z),$$
i.e.\ $J$ can be moved freely among the three arguments of $T$. In particular $T(X,JX,Z)=T(JX,X,Z)=-T(X,JX,Z)$, so $$ T(X,JX,\cdot\,)=0 \qquad \forall X\in C^{\infty}(T\mathbb{S}^{6}).$$
\item[(iv)] Since the nearly K\"ahler structure on $\mathbb{S}^{6}$ is strict,
$\nabla^{LC} J$, and hence $T$, does not vanish at any point of $\mathbb{S}^{6}$.
\end{enumerate}
Now fix $m\in\mathbb{S}^{6}$ and normal coordinates centred at $m$, so that $\Gamma^{k}_{ij}(m)=0$ and $g_{ij}(m)=\delta_{ij}$. Then $$(\nabla^{LC}_{p}J)^{i}_{n}(m)=\partial_{p}J^{i}_{n}(m), \qquad
T_{ijk}:=T(\partial_{i},\partial_{j},\partial_{k})(m)=\partial_{i}J^{k}_{j}(m),$$ and the properties above read, at $m$:
\begin{itemize}
\item by (i), $\partial_{i}J^{k}_{j}$ is skew-symmetric in all three indices; in particular \\ $\partial_{i}J^{k}_{j}+\partial_{j}J^{k}_{i}=0$ and $\sum_{k}\partial_{k}J^{k}_{j}=\sum_{k}T_{kjk}=0$;
\item by (ii), $\sum_{n}\big[(\partial_{p}J^{i}_{n})J^{n}_{k}+J^{i}_{n}(\partial_{p}J^{n}_{k})\big]=0$; taking the trace gives \\ $\operatorname{tr}(J\,\partial_{p}J)=0$, i.e.\
$\sum_{j,k}J^{k}_{j}\,\partial_{p}J^{k}_{j}=0$;
\item by (iii), $\sum_{i,p}J^{p}_{i}\,\partial_{p}J^{i}_{j} =\sum_{i}T(J\partial_{i},\partial_{j},\partial_{i}) =-\sum_{i}T(\partial_{i},J\partial_{i},\partial_{j})=0$.
\end{itemize}
\end{oss}
\begin{lemma}
Let $(\mathbb{S}^{6},J,g)$ be the six-dimensional sphere with its usual nearly K\"ahler structure and let $\nabla ^{LC}$ be the Levi-Civita connection. The weak generalized almost complex structure $\mathsf{J}_{g}$ is $[\cdot ,\cdot ]_{\nabla ^{LC}}$-integrable, whereas $\mathsf{J}_{1,J}$ and $\mathfrak{J}_{\omega}$ are not.
\end{lemma}
\begin{proof}
Let $T$ be as in Remark~\ref{oss:NK} and we consider $\mathsf{J}_{1,J}$. It is not integrable in fact in this case $H=J$, $\alpha=\beta=0$ and $K=J^{*}$. Since $\nabla^{LC}$ is torsion free, (2) reads $$(\nabla_{JX}J)Y-(\nabla_{JY}J)X-J(\nabla_{X}J)Y+J(\nabla_{Y}J)X=0,$$ and the left-hand side is the Nijenhuis tensor $N_{J}(X,Y)$ of $J$. By Remark~\ref{oss:NK}(iii),
$g((\nabla_{JX}J)Y,Z)=T(JX,Y,Z)=T(X,Y,JZ)=-g(J(\nabla_{X}J)Y,Z)$, or $\nabla_{JX}J=-J\nabla_{X}J$ and together with $(\nabla_{Y}J)X=-(\nabla_{X}J)Y$ this gives $$ N_{J}(X,Y)=-4J(\nabla_{X}J)Y , $$
which is not zero by Remark~\ref{oss:NK}(iv).
\\ Now consider $\mathfrak{J}_{\omega}$, here $H=K=0$, $\alpha=-\omega^{-1}$ and $\beta=\omega$. Since $\alpha$ is invertible and $\nabla^{LC}$ is torsion free, (2) gives $(\nabla_{X}\omega)(Y,\cdot\,)-(\nabla_{Y}\omega)(X,\cdot\,)=0$. As $(\nabla_{X}\omega)(Y,Z)=g((\nabla_{X}J)Y,Z)=T(X,Y,Z)$ is totally skew by Remark~\ref{oss:NK}(i), this becomes $2T(X,Y,Z)=0$, i.e.\ $\nabla J=0$, contradicting Remark~\ref{oss:NK}(iv).
\\ Finally $\mathsf{J}_{g}$ is integrable. In this case we have $H=K=0$, $\alpha=-g^{-1}$ and $\beta=g$. Since $\nabla^{LC}$ is symmetric, $\nabla^{LC}g^{-1}=0$ and $\alpha$ is invertible, Corollary 3.0.1 applies.
\end{proof}
\begin{oss}
\cite{N1} Given any metric $g$, the induced weak generalized almost complex structure $\mathsf{J}_{g}$ is always $[\cdot ,\cdot ]_{\nabla ^{LC}}$-integrable.
\end{oss}
\begin{lemma}
Let $(\mathbb{S}^{6},J,g)$ be the six-dimensional sphere with its usual nearly K\"ahler structure and let $\nabla ^{LC}$ be the Levi-Civita connection. Then there is not any spherical combination $\mathsf{J}=a\mathsf{J}_{1,J}+b\mathsf{J}_{g}$ with $a,b \in C^{\infty}(\mathbb{S}^{6})$, and $a^{2}+b^{2}=1$, $a \not\equiv 0$ such that the weak generalized almost complex structure $\mathsf{J}$ is $[\cdot ,\cdot ]_{\nabla ^{LC}}$-integrable.
\end{lemma}
\begin{proof}
Consider a point $m \in \mathbb{S}^{6}$ and choose normal coordinates in a neighbourhood of $m$. Observe that, in these normal coordinates, the Christoffel symbols of the Levi-Civita connection vanish at $m$ and the metric components coincide with the Kronecker delta at $m$, i.e.
$$\Gamma _{ij}^{k}=0, \hspace{0.2cm} g_{ij}=\delta _{ij} \hspace{0.2cm} \forall i,j,k \in \{1,\dots ,6\}.$$
Observe that, in matrix form, $\mathsf{J}$ is given by:
$$\mathsf{J}=\left( \begin{matrix} aJ && -bg^{-1} \\ 
bg && aJ^{*} \end{matrix} \right).$$
From Lemma 3.0.2, we know that $\mathsf{J}_{1,J}$ is not $[\cdot ,\cdot ]_{\nabla ^{LC}}$-integrable, therefore we may consider $b\not \equiv 0$.
Assume, for contradiction, that there exist $a,b$ with $b\not \equiv 0$, $a\not \equiv 0$, such that $\mathsf{J}$ is $[\cdot ,\cdot ]_{\nabla ^{LC}}$-integrable. Therefore, $\mathsf{J}$ satisfies the equations of Proposition 3.0.1, and in particular we may rewrite equation \eqref{eq:3nabla} in normal coordinates:
$$b(\partial _{i}(b\delta^{j}_{k}))-b(\partial _{j}(b\delta_{k}^{i}))=0 \hspace{0.2cm} \forall i,j,k \in \{1, \dots ,6\}.$$
We choose now $i=k\neq j$; we then obtain 
$$b(\partial _{j}b)=0 \hspace{0.2cm} \forall j.$$
It follows that on the open set $\{b\neq 0\}$ we have $\partial _{j}b=0$ for all $j$. Hence $b$ is locally constant on $\{b\neq 0\}$, and therefore on each connected component $U$ of $\{b \neq 0\}$ we have $b\equiv K$ for some constant $K\neq 0$. Let $U$ be such a component. By continuity of $b$, if $x_{n} \in U$ and $x_{n}\rightarrow p \in \overline{U}^{\mathbb{S}^{6}}$, then $b(x_{n})=K$ for all $n$, so $b(p)=K\neq 0$; hence $\overline{U}^{\mathbb{S}^{6}} \subset \{b \neq 0\}$. The set $\overline{U}^{\mathbb{S}^{6}}$ is connected, hence it is a connected subset of $\{b\neq 0\}$ containing $U$. By maximality of $U$ as a connected subset of $\{b\neq 0\}$ we must have $\overline{U}^{\mathbb{S}^{6}}=U$. Therefore $U$ is closed (as well as open) in $\mathbb{S}^{6}$. If $\{b\neq 0\}$ is nonempty then some component $U$ is nonempty and hence clopen in $\mathbb{S}^{6}$; by connectedness of $\mathbb{S}^{6}$ we must have $U=\mathbb{S}^{6}$. Therefore, since $b \not \equiv 0$, it follows that $\partial _{j}b=0$ for all $j$ on all of $\mathbb{S}^{6}$. Then $b$ is globally constant, and consequently $a$ is also constant. Now we consider in normal coordinates equation \eqref{eq:4nabla}:
$$-bg^{ip}(\partial _{p}(aJ_{j}^{k}))+bg^{jp}(\partial _{p}(aJ_{i}^{k}))=0.$$
Since $a \not \equiv 0$, $b\not \equiv 0$ and are constant, we obtain:
$$-\partial _{i}J_{j}^{k}+\partial _{j}J_{i}^{k}=0.$$
As $(\mathbb{S}^{6},J,g)$ is nearly K\"ahler, it follows that $(\nabla _{X}J)X=0$ for all $X\in C^{\infty}(TM)$ and so 
$$(\nabla _{X}J)Y+(\nabla _{Y}J)X=0 \hspace{0.2cm} \forall X,Y \in C^{\infty}(TM),$$
or equivalently, in normal coordinates,
$$\partial _{i}J_{j}^{k}+\partial _{j}J_{i}^{k}=0 \hspace{0.2cm} \forall i,j,k$$
This yields a contradiction, which proves the claim.
\end{proof}

\begin{lemma}
Let $(\mathbb{S}^{6},J,g)$ be the six-dimensional sphere with its usual nearly K\"ahler structure and let $\nabla ^{LC}$ be the Levi-Civita connection. Then there is not any spherical combination $\mathsf{J}=a\mathsf{J}_{1,J}+c\mathfrak{J}_{\omega}$ with $a,c \in C^{\infty}(\mathbb{S}^{6})$, and $a^{2}+c^{2}=1$, such that the weak generalized almost complex structure $\mathsf{J}$ is $\nabla ^{LC}$-integrable.
\end{lemma}
\begin{proof}
Consider a point $m \in \mathbb{S}^{6}$ such that $c(m)\neq 0$, and choose normal coordinates in a neighbourhood of $m$. 
Observe that, in matrix form, $\mathsf{J}$ is given by:
$$\mathsf{J}=\left( \begin{matrix} aJ && -c\omega ^{-1} \\ 
c\omega && aJ^{*} \end{matrix} \right)$$
where $(\omega ^{-1})^{ij}=-g^{ik}J_{k}^{j}$ and $\omega _{ij}=J_{i}^{k}g_{kj}.$
\\ Assume, for contradiction, that there exist $a, c$ such that $\mathsf{J}$ is $[\cdot ,\cdot ]_{\nabla ^{LC}}$-integrable. Therefore, $\mathsf{J}$ satisfies the equations of Proposition 3.0.1, and in particular we may rewrite equation \eqref{eq:3nabla} in normal coordinates:
$$cg^{il}J_{l}^{p}(\partial _{p}(cg^{jm}J_{m}^{k}))-cg^{jl}J_{l}^{p}(\partial _{p}(cg^{im}J_{m}^{k}))=0.$$
Since $\alpha_{i}^{i}=0$ for all $i$, we set $k=i$; the previous equation then reads:
$$c\delta ^{il}J_{l}^{p}(\partial _{p}(c\delta ^{jm}J_{m}^{k}))=0.$$
Observe that neither $a$ nor $c$ can be identically zero, since neither $\mathsf{J}_{1,J}$ nor $\mathfrak{J}_{\omega}$ is $[\cdot ,\cdot ]_{\nabla ^{LC}}$-integrable by Lemma 3.0.2. Therefore, the equation can be rewritten, making the Einstein summation convention explicit, as follows:
$$\sum _{p=1}^{6}J_{i}^{p}(\partial _{p}(cJ_{j}^{i}))=0$$
or
$$\sum _{p=1}^{6}(J_{i}^{p}J_{j}^{i}\partial _{p}c+cJ_{i}^{p}\partial _{p}J_{j}^{i})=0.$$
Since the identity holds for each index $i$, summing over $i$ and, after introducing $S_{j}=\sum _{p,i=1}^{6}J_{i}^{p}\partial _{p}J_{j}^{i}$ and noting that $\sum _{i}J_{i}^{p}J_{j}^{i}=-\delta _{j}^{p}$, the equation may be rewritten as follows:
$$\partial _{j}c=cS_{j}.$$
Since $J$ is an almost complex structure, one has $J^{2}=-Id$, from which it follows that $$(\partial _{p}J_{k}^{i})J_{j}^{k}+J_{k}^{i}(\partial _{p}J_{j}^{k})=0.$$
Contracting this equation with $J_{i}^{p}$ (i.e. multiplying by $J_{i}^{p}$ and summing over the indices $i,p$), we obtain, making the Einstein summation convention explicit: 
$$\sum _{i,p} J_{i}^{p}(\partial _{p}J_{k}^{i})J_{j}^{k}+\sum _{i,p}J_{i}^{p}J_{k}^{i}(\partial _{p}J_{j}^{k})=0$$
that is 
$$S_{k}J_{j}^{k}=\delta _{k}^{p}(\partial _{p}J_{j}^{k}).$$
Finally, observe that, since the structure is nearly K\"ahler, $\partial _{k}J_{j}^{k}=0$ for all $j$, $k$ and therefore $S_{k}=0$ for all $k$. Substituting into the previous equation, we obtain $\partial _{j}c=0$ for all $j$ and hence, by the arbitrariness of $m \in \{c \neq 0\}$ and the continuity of $c$ on the connected manifold $\mathbb{S}^{6}$, using the clopen argument as in the previous proof, it follows that $c$ and $a$ are constant on all of $\mathbb{S}^{6}$.
We now show that, with $a$ and $c$ nonzero constants, $\mathsf{J}$ cannot be $[\cdot,\cdot]_{\nabla^{LC}}$-integrable. Recall that $\omega(X,\cdot\,)=g(JX,\cdot\,)$, i.e.\ $\omega=\flat\circ J$ as a map $T\mathbb{S}^{6}\to T^{*}\mathbb{S}^{6}$; hence $\omega^{-1}=-J\circ\sharp$ and the blocks of $\mathsf{J}$ are $$\alpha=-c\,\omega^{-1}=c\,J\circ\sharp, \qquad K=aJ^{*},$$
that is, $\alpha(Y^{\flat})=c\,JY$ for every $Y\in C^{\infty}(T\mathbb{S}^{6})$. Moreover, for every $\eta\in C^{\infty}(T^{*}\mathbb{S}^{6})$ and $X,Z\in C^{\infty}(T\mathbb{S}^{6})$,
$$\big((\nabla_{X}J^{*})\eta\big)(Z) = X\big(\eta(JZ)\big)-(\nabla_{X}\eta)(JZ)-\eta(J\nabla_{X}Z) = \eta\big((\nabla_{X}J)Z\big),$$
so that, $a$ being constant, $(\nabla_{X}K)\eta=a\,\eta\circ(\nabla_{X}J)$. By Proposition 3.0.1, the integrability of $\mathsf{J}$ implies (5):
$$(\nabla_{\alpha\xi}K)\eta-(\nabla_{\alpha\eta}K)\xi=0 \qquad \forall\,\xi,\eta\in C^{\infty}(T^{*}\mathbb{S}^{6}).$$
Since $g$ is nondegenerate, we may take $\xi=X^{\flat}$, $\eta=Y^{\flat}$; then $\alpha\xi=cJX$ and $\alpha\eta=cJY$. Evaluating on $Z$ and using the $C^{\infty}$-linearity of $\nabla_{V}$ in $V$, we get $$ 0=ac\Big[g\big((\nabla_{JX}J)Z,Y\big)-g\big((\nabla_{JY}J)Z,X\big)\Big] =ac\,\big[T(JX,Z,Y)-T(JY,Z,X)\big].$$
By Remark~\ref{oss:NK} (i) we have $T(JX,Z,Y)=-T(JX,Y,Z)$, while by Remark~\ref{oss:NK} (i) and (iii) $$T(JY,Z,X)=-T(JY,X,Z)=-T(Y,JX,Z)=T(JX,Y,Z).$$ Therefore (5) becomes
$$-2ac\,T(JX,Y,Z)=0 \qquad \forall\,X,Y,Z\in C^{\infty}(T\mathbb{S}^{6}).$$ Since $ac\neq 0$ and $J$ is invertible (given any $W$, take $X=-JW$, so that $JX=W$), it follows that $T\equiv 0$, i.e.\ $\nabla J=0$ on $\mathbb{S}^{6}$. This contradicts Remark~\ref{oss:NK} (iv) and proves the claim.
\end{proof}
\begin{proposition}
Let $(\mathbb{S}^{6},J,g)$ be the six-dimensional sphere with its usual nearly K\"ahler structure and let $\nabla ^{LC}$ be the Levi-Civita connection. Then there is not any spherical combination $\mathsf{J}=a\mathsf{J}_{1,J}+b\mathsf{J}_{g}+c\mathfrak{J}_{\omega}$ with $a,b,c \in C^{\infty}(\mathbb{S}^{6})$, $a^{2}+b^{2}+c^{2}=1$, and $b \not\equiv \pm 1$ such that the weak generalized almost complex structure $\mathsf{J}$ is $[\cdot ,\cdot ]_{\nabla ^{LC}}$-integrable.
\end{proposition}
\begin{proof}
Consider a point $m \in \mathbb{S}^{6}$ and choose normal coordinates in a neighbourhood of $m$. 
Observe that, in matrix form, $\mathsf{J}$ is given by:
$$\mathsf{J}=\left( \begin{matrix} aJ && -c\omega ^{-1} -bg^{-1} \\ 
c\omega +bg && aJ^{*} \end{matrix} \right)$$
where $(\omega ^{-1})^{ij}=-g^{ik}J_{k}^{j}$ and $\omega _{ij}=J_{i}^{k}g_{kj}.$
\\ Assume, for contradiction, that there exist $a,b,c$ with $b\not \equiv \pm 1$ such that $\mathsf{J}$ is $[\cdot ,\cdot ]_{\nabla ^{LC}}$-integrable. Therefore, $\mathsf{J}$ satisfies the equations of Proposition 3.0.1, and in particular we may rewrite equation \eqref{eq:3nabla} in normal coordinates:
$$(-bg^{ip}+cg^{il}J_{l}^{p})(\partial _{p}(-bg^{jk}+cg^{jm}J_{m}^{k}))+$$
$$-(-bg^{jp}+cg^{jl}J_{l}^{p})(\partial _{p}(-bg^{ik}+cg^{im}J_{m}^{k}))=0.$$
Observing that in normal coordinates and that in this case $\alpha ^{ip}=-b\delta ^{ip}+cJ_{i}^{p}$, we can rewrite the equation as:
$$\alpha ^{ip}(-\delta ^{jk}\partial _{p}b+J_{j}^{k}\partial _{p}c+c\partial _{p}J_{j}^{k})-\alpha ^{jp}(-\delta ^{ik}\partial _{p}b+J_{i}^{k}\partial _{p}c+c\partial _{p}J_{i}^{k})=0.$$
Now we define:
$$A^{i}=\sum_{p}\alpha ^{ip}\partial _{p}b, \hspace{0.4cm} C^{i}=\sum _{p}\alpha ^{ip}\partial _{p}c, \hspace{0.4cm} D^{ijk}=\sum _{p}\alpha ^{ip}\partial _{p}J^{jk}$$
and substitute these into the equation to obtain
\begin{equation} \label{eq:1dim}
-A^{i}\delta ^{jk}+C^{i}J^{jk}+cD^{ijk}+A^{j}\delta ^{ik}-C^{j}J^{ik}-cD^{jik}=0.
\end{equation}
Observe first that $C^{i}J^{jk}$ and $D^{ijk}$ are skew-symmetric in the indices $j$, $k$. Hence we symmetrize the previous identity in $j$, $k$:
$$-A^{i}\delta ^{jk}+\frac{1}{2}(A^{j}\delta ^{ik}+A^{k}\delta ^{ij})-\frac{1}{2}(C^{j}J^{ik}+C^{k}J^{ij})-\frac{1}{2}c(D^{jik}+D^{kij})=0.$$
Contracting this equation with $\delta _{jk}$ gives:
$$\sum_{j,k}\left(-A^{i}\delta _{jk}\delta ^{jk}+\frac{1}{2}\delta _{jk}(A^{j}\delta ^{ik}+A^{k}\delta ^{ij})-\frac{1}{2}\delta _{jk}(C^{j}J^{ik}+C^{k}J^{ij})-\frac{1}{2}c\delta _{jk}(D^{jk}+D^{kij})\right) =0.$$
Carrying out the contractions yields
\begin{equation} \label{eq:2dim}
-6A^{i}+A^{i}-\sum_{j}C^{j}J^{ij}-c\sum _{j}D^{jij}=0.
\end{equation}
Using the definition of $D^{ijk}$ we have 
$$D^{jij}=-b\partial _{j}J^{ij}+c\sum _{p}J^{jp}\partial _{p}J^{ij}$$
and by Remark~\ref{oss:NK} $\partial _{j}J^{ij}=0.$ Moreover, from 
$$(\partial _{p}J_{n}^{i})J_{k}^{n}+J_{n}^{i}(\partial _{p}J_{k}^{n})=0$$
(if we set $k=p$ and sum over $p$) we obtain
$$\sum _{p,n}(\partial _{p}J_{n}^{i})J_{p}^{n}=0$$
which implies $\sum _{j}D^{jij}=0$. Therefore equation \eqref{eq:2dim} reduces to:
$$(J\cdot C)^{i}=-5A^{i}.$$
i.e. $JC=-5A.$
\\ Next, start from equation \eqref{eq:1dim} and contract it with $J_{jk}$. This yields
$$\sum _{j,k}(-A^{i}J_{jk}\delta ^{jk}+C^{i}J_{jk}J^{jk}+cJ_{jk}D^{ijk}+J_{jk}A^{j}\delta ^{ik}-J_{jk}C^{j}J^{ik}-cJ_{jk}D^{jik})=0$$
which after simplifying becomes
$$-A^{i}J_{j}^{j}-6C^{i}+c\sum _{j,k}J_{jk}D^{ijk}+\sum _{j}J_{ji}A^{j}-\sum _{j,k}J^{ik}J_{jk}C^{j}-c\sum _{j,k}J_{jk}D^{jik}=0.$$
By Remark~\ref{oss:NK},
$$\sum _{j,k}J_{jk}D^{ijk}=\sum _{j,k,p}\alpha ^{ip}J_{jk}\partial _{p}J^{jk}=\sum _{p}\alpha ^{ip}\sum _{j,k}J_{jk}\partial _{p}J^{jk}=0.$$
Furthermore $-\sum _{k,j}J^{ik}J_{jk}C^{j}=\delta _{j}^{i}C^{j}=C^{i}$ and arguing as before for $\sum _{j}D^{jij}$ one shows
$$c\sum _{j,k}J_{jk}D^{jik}=c\sum_{j,k}J_{jk}\sum _{p}\alpha ^{ip}\partial _{p}J^{ik}=$$
$$=c(\sum _{p,k}-bJ_{pk}\partial _{p}J^{ik}+\sum _{p,k}c\delta _{k}^{p}\partial _{p}J^{ik})=-bc\sum _{p,k}J_{pk}\partial _{p}J^{ik}=0.$$
Substituting these vanishing terms into the contracted equation gives:
$$-6C^{i}+(AJ)^{i}+C^{i}=0$$
i.e. $JA=5C$. Thus we have the linear system
$$ \begin{cases}
JC=-5A \\ 
5C=JA 
\end{cases}$$
from which it follows that $C=A=0$. Since $A$ and $C$ can be written respectively as $\alpha \nabla^{LC} b=0$ and $\alpha \nabla^{LC} c=0$; since $\alpha$ is invertible in $\Omega =\mathbb{S}^{6}-\{x \in \mathbb{S}^{6}:b(x)=c(x)=0\}$, we have $\nabla^{LC} b=\nabla^{LC} c=0$ in $\Omega$. Therefore $b$ and $c$ are locally constant on $\Omega$, and consequently on each connected component $U$ of $\Omega$ they take constant values, say $b\equiv K$ and $c\equiv K_{2}$ on $U$, with $(K,K_2)\neq (0,0)$. Let $U$ be such a component. By continuity of $b$ and $c$, any limit point of $U$ belongs to $\Omega$ and therefore $b$ and $c$ assume the same constant values $K$, $K_{2}$ on $\overline{U}^{\mathbb{S}^{6}}$; in particular $\overline {U}^{\mathbb{S}^{6}}\subset \Omega$ and it is connected. Hence, by maximality of $U$ as a connected subset of $\Omega $, we must have $U=\overline{U}^{\mathbb{S}^{6}}$. Therefore $U$ is closed (as well as open) in $\mathbb{S}^{6}$. If $\Omega$ is nonempty then some component $U$ is nonempty and therefore clopen in $\mathbb{S}^{6}$; by connectedness of $\mathbb{S}^{6}$ we must have $U=\mathbb{S}^{6}$. Therefore $b$ and $c$ are constant (and nonzero for the results in previous lemmas) on all of $\mathbb{S}^{6}$. In particular the endomorphism $\alpha $ is everywhere invertible and $\nabla^{LC} b=\nabla^{LC} c=0$ on $\mathbb{S}^{6}$. Hence, $a$, $b$ and $c$ are globally constant. The empty case is resolved by the Lemma 3.0.2.  \\ We already observed in the preceding lemmas that $b\not \equiv 0$ and $c\not \equiv 0$, since either vanishing identically would yield structures that are not $[\cdot ,\cdot ]_{\nabla^{LC}}$-integrable. Assume now $a \not \equiv 0$. Consider equation \eqref{eq:4nabla} specialized to $\mathsf{J}$:
$$(-bg^{ip}+cg^{il}J_{l}^{p})(\partial _{p}(aJ_{j}^{k}))-(-bg^{jp}+cg^{jl}J_{l}^{p})(\partial _{p}(aJ_{i}^{k}))=0.$$
Expanding and simplifying yields
$$-2b\partial _{i}J_{k}^{j}+c(\sum _{p}J^{ip}\partial _{p}J_{j}^{k}-\sum _{p}J^{jp}\partial _{p}J_{i}^{k})=0.$$
Using $(\nabla^{LC} _{X}J)J+J(\nabla^{LC} _{X}J)=0$ and Remark~\ref{oss:NK} in normal coordinates, the previous identity can be rewritten as
$$-2b\partial _{i}J_{k}^{j}-2c(\partial _{i}J_{j}^{n})J_{n}^{k}=0$$
or, more compactly,
$$\partial _{i}J_{j}^{n}(b\delta _{n}^{k}+cJ_{n}^{k})=0.$$
Writing this in endomorphism form with $B=bId+cJ$:
$$(\partial _{i}J_{j})^{n}B_{n}^{k}=0.$$
Since $B$ is invertible (because $b$ and $c$ are not both zero) we would deduce $\partial _{i}J_{j}^{n}=0$ for all $i,j,n$ contradicting the nearly K\"ahler condition $\nabla^{LC} J \neq 0$. Hence $a \equiv 0$.
Finally, even when $a \equiv 0$ one reaches a contradiction. Indeed, write equation \eqref{eq:1nabla} in normal coordinates for $\mathsf{J}$:
$$(-b\delta ^{lk}+c\delta ^{ln}J_{n}^{k})(\partial _{i}(b\delta ^{jl}+cJ_{j}^{n}\delta _{nl})-\partial _{j}(b\delta ^{il}+cJ_{i}^{n}\delta _{nl}))=0$$
which, using Remark~\ref{oss:NK}, reduces to:
$$2c(-b\delta ^{lk}+cJ_{l}^{k})\partial _{i}J_{j}^{l}=0.$$
Again, viewing this in terms of endomorphism and recalling that $-bI+cJ$ is invertible and $c\neq 0$, we would be forced to conclude $\nabla^{LC} J=0$, a contradiction.
\\ This completes the proof.
\end{proof}

\section{Courant-integrability}
In this section we partially revisit the material presented in the previous section. Specifically, we first characterize integrability with respect to the Courant bracket, both intrinsically and in local coordinates. Among the spherical combinations, the only strong ones are $\pm\mathfrak{J}_{\omega}$ (Proposition 2.9.3). Since $\pm\mathfrak{J}_{\omega}$ is Courant integrable if and
only if $d\omega=0$, these are not Courant integrable on $\mathbb{S}^{6}$ (Lemma 4.0.1). Therefore, it will be necessary to consider other classes of generalized almost complex structures. Furthermore, $\mathbb{S}^{6}$ admits no generalized almost complex structures of constant type $1$ or $2$, and no integrable ones of type $0$.
\begin{lemma}
The sphere $\mathbb{S}^{6}$ does not admit generalized almost complex structures of constant type $1$ or $2$. Furthermore, it does not admit integrable generalized complex structures of type $0$.
\end{lemma}
\begin{proof}
Let $V_{i}\subset E\otimes\mathbb{C}$ be the $+i$-eigenbundle of $\mathfrak{J}$ and let $\Delta:=\pi(V_{i})\subset TM\otimes\mathbb{C}$. Then (\cite{G1}) $V_{i}=L(\Delta,\varepsilon)=\{X+\xi\;:\;X\in\Delta,\ \xi|_{\Delta}=\iota_{X}\varepsilon\}$ for some $\varepsilon\in\wedge^{2}\Delta^{*}$, and the type of $\mathfrak{J}$ at $x$ is $k=\operatorname{codim}_{\mathbb{C}}\Delta_{x}$. Moreover $\Delta\cap\overline{\Delta}=D\otimes\mathbb{C}$ for a real subspace $D\subset TM$ of dimension $2n-2k$ (here $2n=\dim M$). The form $\operatorname{Im}\varepsilon$ restricts to a nondegenerate $2$-form on $D$, and, since $\Delta + \overline{\Delta}=TM\otimes\mathbb{C}$, the quotient $N:=TM/D$ inherits a complex structure, with $N^{0,1}\cong\Delta/(D\otimes\mathbb{C})$. If the type is constant, $D$ and $N$ are smooth vector bundles. No integrability is needed for any of this.
\\ \emph{Types $1$ and $2$.} Suppose, by contradiction, that $\mathbb{S}^{6}$ carries a generalized almost complex structure of constant type $k\in\{1,2\}$. Choosing a
Riemannian metric, we identify $N$ with $D^{\perp}$ and obtain a splitting $$ T\mathbb{S}^{6}\cong D\oplus N, \qquad \operatorname{rk}_{\mathbb{R}}D=6-2k,\quad \operatorname{rk}_{\mathbb{R}}N=2k.$$ Both summands are orientable because $D$ carries a nondegenerate $2$-form and $N$ a complex structure.
\\ Let us consider the Euler class $e(T\mathbb{S}^{6})$ and observe that (\cite{Chern1944}) we obtain:
$$\int _{\mathbb{S}^{6}}e(T\mathbb{S}^{6})=\chi (\mathbb{S}^{6})=2.$$
At the same time, however, the Euler class can be decomposed via the cup product as:
$$ e(T\mathbb{S}^{6})=\pm\, e(D)\smile e(N), \qquad e(D)\in H^{6-2k}(\mathbb{S}^{6};\mathbb{Z}),\quad e(N)\in H^{2k}(\mathbb{S}^{6};\mathbb{Z}).$$
For $k\in\{1,2\}$ both degrees belong to $\{2,4\}$, and $H^{2}(\mathbb{S}^{6};\mathbb{Z})=H^{4}(\mathbb{S}^{6};\mathbb{Z})=0$ and therefore $e(T\mathbb{S}^{6})=0$, which yields a contradiction.
\\ \emph{Type $0$} If $\mathfrak{J}$ has type $0$, then $\Delta=TM\otimes\mathbb{C}$ at every point, so $\varepsilon$ is a complex $2$-form on $M$, which we write as $\varepsilon=B+i\omega$ with $B,\omega\in\Omega^{2}(M)$ and $\omega$ nondegenerate. In particular, $V_{i}=\{X+\iota_{X}(B+i\omega)\;:\;X\in TM\otimes\mathbb{C}\}$ is the graph of $B+i\omega$. The graph of a $2$-form is Courant involutive if and only if the form is closed and so if $\mathfrak{J}$ is integrable, $d(B+i\omega)=0$ and in particular $d\omega=0$, or\ $\omega$ is a symplectic form on $\mathbb{S}^{6}$. But since $H^{2}(\mathbb{S}^{6},\mathbb{R})=0$, there would exist a 1-form $\alpha$ such that $\omega=d\alpha$, and therefore the volume form could be written as $\omega^{3}=d(\alpha\wedge\omega\wedge\omega)$. By Stokes' theorem, it follows that:
$$\int _{\mathbb{S}^{6}}\omega ^{3}=\int _{\mathbb{S}^{6}} d(\alpha \wedge \omega \wedge \omega )=\int _{\partial \mathbb{S}^{6}}(\alpha \wedge \omega \wedge \omega )=0$$
which would imply $\int \omega^{3}=0$, an absurdity since $\omega$ is a non-degenerate symplectic form and $\omega ^{3}$ a volume form.
\end{proof}
\begin{oss}
Finally, we observe that a generalized almost complex structure of type $3$ (complex type) can exist on the sphere $\mathbb{S}^{6}$, since it admits almost complex structures. However, its integrability is strictly tied to the Hopf Problem. Indeed, a generalized almost complex structure of the form $\mathfrak{J}_{J}$ is integrable with respect to the Courant bracket if and only if $J$ is integrable.
\end{oss}
\subsection{Integrability conditions}
\begin{lemma}
Let $\mathfrak{J}=\left( \begin{matrix} H && \alpha \\ \beta && K \end{matrix} \right) $ be a strong generalized almost complex structure on $M$. Let $N(\mathfrak{J})$ be the Nijenhuis tensor of $\mathfrak{J}$ with respect to $[\cdot ,\cdot ]_{C}$.
Then $\mathfrak{J}$ is $[\cdot ,\cdot ]_{C}$-integrable if and only if the following conditions hold for all $X,Y \in C^{\infty}(TM)$ and for all $\xi , \eta \in C^{\infty}(T^{*}M)$:
\begin{equation}
\begin{aligned}
N_{\mathfrak{J}}(X,Y)\big|_{C^{\infty}(TM)}
&\quad =[HX,HY]-H([HX,Y]+[X,HY]) -\alpha (-\mathfrak{L}_{Y}\beta X \\ 
&\quad +\mathfrak{L}_{X}\beta Y+\frac{1}{2}d(\beta X(Y))-\frac{1}{2}d(\beta Y(X)))-[X,Y]=0,
\end{aligned}
\end{equation}

\begin{equation}
\begin{aligned}
N_{\mathfrak{J}}(X,Y)\big|_{C^{\infty}(T^{*}M)}=
&\quad \mathfrak{L}_{HX}(\beta Y)-\mathfrak{L}_{HY}(\beta X)-\frac{1}{2}d((\beta Y)(HX)+ \\ 
&\quad -(\beta X)(HY))-\beta ([HX,Y]+[X,HY])-K(-\mathfrak{L}_{Y}\beta X+ \\ 
&\quad +\mathfrak{L}_{X}\beta Y +\frac{1}{2}d(\beta X(Y))-\frac{1}{2}d(\beta Y(X)))=0,
\end{aligned}
\end{equation}

\begin{equation}
\begin{aligned}
N_{\mathfrak{J}}(\xi ,\eta )\big| _{C^{\infty }(TM)}
&\quad =[\alpha \xi ,\alpha \eta ]-\alpha (\mathfrak{L}_{\alpha \xi }\eta -\mathfrak{L}_{\alpha \eta }\xi +\\ &\quad +\frac{1}{2}d(-\eta (\alpha \xi )+\xi (\alpha \eta )))=0,
\end{aligned}
\end{equation}

\begin{equation}
\begin{aligned}
N_{\mathfrak{J}}(\xi ,\eta )\big| _{C^{\infty}(T^{*}M)}
&\quad =\mathfrak{L}_{\alpha \xi }(K\eta )-\mathfrak{L} _{\alpha \eta }(K\xi )-\frac{1}{2}d(-(K\xi )(\alpha \eta )+ \\ 
&\quad +(K\eta )(\alpha \xi ))-K(\mathfrak{L}_{\alpha \xi }\eta -\mathfrak{L}_{\alpha \eta }\xi +\\ 
&\quad +\frac{1}{2}d(-\eta (\alpha \xi )+\xi (\alpha \eta )))=0,
\end{aligned}
\end{equation}

\begin{equation}
\begin{aligned}
N_{\mathfrak{J}}(X,\eta )\big| _{C^{\infty}(TM)}
&\quad =[HX,\alpha \eta ]-H([X,\alpha \eta ])-\alpha (\mathfrak{L}_{HX}\eta +\mathfrak{L}_{X}(K\eta )+ \\ 
&\quad -\frac{1}{2}d(\eta (HX)+K\eta (X)))=0,
\end{aligned}
\end{equation}

\begin{equation}
\begin{aligned}
N_{\mathfrak{J}}(X,\eta )\big|_{C^{\infty } (T^{*}M)}
&\quad =\mathfrak{L}_{HX}(K\eta )-\mathfrak{L}_{\alpha \eta}(\beta X)+\frac{1}{2}d((\beta X)(\alpha \eta )+ \\ 
&\quad -(K\eta )(HX))-\beta ([X,\alpha \eta ])-K(\mathfrak{L}_{HX}\eta +\mathfrak{L}_{X}(K\eta )+\\ 
&\quad -\frac{1}{2}d(\eta (HX)+K\eta (X))) -\mathfrak{L}_{X}\eta +\frac{1}{2}d(\eta (X))=0.
\end{aligned}
\end{equation}
\end{lemma}
\begin{proof}
Recall that a strong generalized almost complex structure is said to be $[\cdot ,\cdot ]_{C}$-integrable if and only if its Nijenhuis tensor vanishes identically, i.e.
$$N_{\mathfrak{J}}(X+\xi ,Y+\eta )=0 \hspace{0.4cm} \forall (X+\xi ) ,(Y+\eta ) \in C^{\infty}(E).$$
Carrying out the computations using Definition $2.8$ with $\mathfrak{J}$ and the bracket $[\cdot ,\cdot ]_{C}$ yields the six conditions stated above. For example, we now compute condition $(16)$: 
$$N_{\mathfrak{J}}(X,Y)\big|_{C^{\infty}(TM)}=([\mathfrak{J}X,\mathfrak{J}Y]_{C}-\mathfrak{J}[\mathfrak{J}X,Y]_{C}-\mathfrak{J}[X,\mathfrak{J}Y]_{C}- [X,Y]_{C})\big|_{C^{\infty}(TM)}=$$
$$=([HX,HY]+\mathfrak{L}_{HX}(\beta Y)-\mathfrak{L}_{HY}(\beta X)-\frac{1}{2}d((\beta Y)(HX)-(\beta X)(HY))+$$
$$-(H+\beta )([HX,Y]+[X,HY])-(\alpha +K)(-\mathfrak{L}_{Y}(\beta X)+\mathfrak{L}_{X}(\beta Y)+$$
$$+\frac{1}{2}d((\beta X)(Y)-(\beta Y)(X)))-[X,Y])\big|_{C^{\infty}(TM)}=$$
$$=[HX,HY]-H([HX,Y]+[X,HY])-\alpha (-\mathfrak{L}_{Y}(\beta X)+$$
$$+\mathfrak{L}_{X}(\beta Y)+\frac{1}{2}d(\beta X(Y)-\beta Y(X))-[X,Y].$$
\end{proof}

\begin{lemma}
Let $\mathfrak{J}=\left( \begin{matrix} H && \alpha \\ \beta && K \end{matrix} \right) $ be a strong generalized almost complex structure on $M$. Let us take local coordinates $(U,(x^{1},\dots ,x^{n}))$ such that:
$$H\partial _{i}=H^{j}_{i}\partial _{j}, \hspace{0.4cm} \alpha dx^{i}=\alpha ^{ij}\partial_{j}, \hspace{0.4cm}
\beta \partial _{i}=\beta _{ij} dx^{j}, \hspace{0.4cm}
K dx^{i}=K_{j}^{i}dx^{j}.$$
Then $\mathfrak{J}$ is Courant-integrable if and only if the following conditions are met for all $i,j,k \in \{1, \dots , dim_{\mathbb{R}}M\}$:
\begin{equation}
\begin{aligned}
(N_{\mathfrak{J}}(\partial _{i},\partial _{j})\big|_{C^{\infty}(TM)})^{k}
&\quad =H_{i}^{l}\partial_{l}H_{j}^{k}-H_{j}^{l}\partial _{l}H_{i}^{k}+H_{l}^{k}(\partial _{j}H_{i}^{l}-\partial_{i}H_{j}^{l})+\\ 
&\quad +\alpha ^{lk}(\partial _{j}\beta _{il}-\partial _{i}\beta _{jl}-\frac{1}{2}\partial _{l}(\beta _{ij}-\beta _{ji}))=0,
\end{aligned}
\end{equation}

\begin{equation}
\begin{aligned}
(N_{\mathfrak{J}}(\partial _{i},\partial _{j})\big|_{C^{\infty}(T^{*}M)})^{k}
&\quad = H_{i}^{l}\partial _{l}\beta _{jk}+\beta _{jl}\partial _{k}H_{i}^{l}-H_{j}^{l}\partial _{l}\beta _{ik}-\beta _{il}\partial _{k}H_{j}^{l}+ \\ 
&\quad -\frac{1}{2}\partial _{k}(\beta _{jl}H_{i}^{l}-\beta _{il}H_{j}^{l})+\beta _{lk}(\partial _{j}H_{i}^{l}-\partial _{i}H_{j}^{l})+\\ 
&\quad +K_{k}^{l}(\partial _{j}\beta _{il}-\partial _{i}\beta _{jl}-\frac{1}{2}\partial _{l}(\beta _{ij}-\beta _{ji}))=0,
\end{aligned}
\end{equation}

\begin{equation}
\begin{aligned}
(N_{\mathfrak{J}}(dx^{i},dx^{j})\big|_{C^{\infty}(TM)})^{k}
&\quad =\alpha ^{il}\partial _{l}\alpha ^{jk}-\alpha ^{jl}\partial _{l}\alpha ^{ik}-\frac{1}{2}\alpha ^{lk}(\partial _{l}(\alpha ^{ij}-\alpha ^{ji}))=0
\end{aligned}
\end{equation}

\begin{equation}
\begin{aligned}
(N_{\mathfrak{J}}(dx^{i},dx^{j})\big|_{C^{\infty}(T^{*}M)})^{k}
&\quad =\alpha ^{il}\partial _{l}K_{k}^{j}+K_{l}^{j}\partial _{k}\alpha ^{il}-\alpha ^{jl}\partial _{l}K_{k}^{i}-K_{l}^{i}\partial _{k}\alpha ^{jl}+\\ 
&\quad +\frac{1}{2}(\partial _{k}(K_{l}^{i}\alpha ^{jl}-K_{l}^{j}\alpha ^{il}))-\frac{1}{2}K_{k}^{l}(\partial _{l}(\alpha ^{ij}-\alpha ^{ji}))=0,
\end{aligned}
\end{equation}
\begin{equation}
\begin{aligned}
(N_{\mathfrak{J}}(\partial _{i},dx^{j})\big|_{C^{\infty}(TM)})^{k}
&\quad =H_{i}^{l}\partial _{l}\alpha ^{jk}-\alpha ^{jl}\partial _{l}H_{i}^{k}-H_{l}^{k}\partial _{i}\alpha ^{jl}+ \\ 
&\quad -\alpha ^{lk}(\partial _{i}K_{l}^{j}+\frac{1}{2}\partial _{l}H_{i}^{j}-\frac{1}{2}\partial _{l}K_{i}^{j})=0,
\end{aligned}
\end{equation}
\begin{equation}
\begin{aligned}
(N_{\mathfrak{J}}(\partial _{i},dx^{j})\big|_{C^{\infty}(T^{*}M)})^{k}
&\quad = H_{i}^{l}\partial _{l}K_{k}^{j}+K_{l}^{j}\partial _{k}H_{i}^{l}-\alpha ^{jl}\partial _{l}\beta _{ik}-\beta _{il}\partial _{k}\alpha ^{jl}+\\ 
&\quad +\frac{1}{2}\partial _{k}(\beta _{il}\alpha ^{jl}-K_{l}^{j}H_{i}^{l})-\beta _{lk}\partial _{i}\alpha ^{jl}-K_{k}^{l}(\partial _{i}K_{l}^{j}+\\ 
&\quad +\frac{1}{2}\partial _{l}H_{i}^{j}-\frac{1}{2}\partial _{l}K_{i}^{j})=0.
\end{aligned}
\end{equation}
\end{lemma}
\begin{proof}
The proof proceeds by expressing the conditions of the preceding lemma in local coordinates. One then uses the fundamental properties of the Lie bracket together with the following identities:
$$(\mathfrak{L}_{X}\xi )_{m}=X^{k}\partial _{k}\xi _{m}+\xi _{k}\partial _{m}X^{k},$$
$$[f^{k}\partial _{k} ,g^{l}\partial _{l}]=(f^{k}\partial _{k}g^{l}-g^{k}\partial _{k}f^{l})\partial _{l}+f^{k}g^{l}[\partial _{k},\partial _{l}],$$
$$d(\iota _{X}\eta )=\partial _{k}(X^{j}\eta _{j})dx^{k}.$$
\end{proof}
\begin{es}
We know that strong generalized almost complex structures of complex type are integrable if and only if the almost complex structures defining them is. For example, consider the standard complex structure $J_{0}$ on $\mathbb{R}^{2n}$. This is integrable and therefore the generalized almost complex structure $\mathfrak{J}_{J_{0}}$ is Courant-integrable. Indeed, one easily checks that, since $\alpha = \beta =0$ and $J_{0}$ is constant in certain global charts, one has $\partial _{i}J_{0}=0$ and therefore all six conditions of the previous lemma are satisfied.
\end{es}

\subsection{Obstructions to Gluing Local Models on $\mathbb{S}^{6}$}
As previously observed, the only strong spherical combinations, $\pm\mathfrak{J}_{\omega}$, are not Courant integrable on $\mathbb{S}^{6}$. Moreover, by Lemma 4.0.1, $\mathbb{S}^{6}$ admits no generalized almost complex structures of constant type $1$ or $2$ and no generalized complex structures of type $0$. It is well known \cite{Sev} that exact Courant algebroids over $M$ are
classified, up to isomorphism, by the \v{S}evera class $[H]\in H^{3}(M,\mathbb{R})$. In particular, since $H^{3}(\mathbb{S}^{6},\mathbb{R})=0$, every exact Courant algebroid over $\mathbb{S}^{6}$ is isomorphic to the untwisted one $(T\mathbb{S}^{6}\oplus T^{*}\mathbb{S}^{6},[\cdot,\cdot]_{C})$. Proposition 4.0.1 is an explicit instance of this construction, which we include for completeness.
\begin{proposition}
Let $M$ be a smooth manifold admitting a finite open cover $\{ U_{i} \}$. For each $U_{i}$, let $E_{i}$ denote the generalized tangent bundle over $U_{i}$. Let $\{ \varphi _{ij} \}$ be a family of smooth maps defined on the overlaps $U_{ij}=U_{i}\cap U_{j}$,
$$\varphi _{ij} :E_{j}\big|_{U_{ij}}\rightarrow E_{i}\big|_{U_{ij}}$$
satisfying the following conditions:
\begin{nicenum}[label=\roman*.]
\item Each $\varphi _{ij}$ is a $b$-field-transform.
\item $\varphi _{ii}=Id$ and $\varphi _{ij}=\varphi _{ji}^{-1}$.
\item On the triple intersections $U_{ijk}=U_{i}\cap U_{j}\cap U_{k}$, one has $\varphi _{ij}\circ \varphi _{jk}=\varphi _{ik}$ for all $i,j,k $.
\end{nicenum}
Therefore:
\begin{nicenum}
\item It is possible to construct the vector bundle $E$ by gluing of $E_{i}$. On the space of sections of $E$ one can define a pairing $\langle ,\rangle $, an anchor $\rho $, and a bracket which make $E$ into a Courant algebroid.
\item Since the $\varphi _{ij}$ are $b$-field-transforms, they preserve the cotangent subbundle $T^{*}M$. Hence $E$ is an exact Courant algebroid. Consequently, there exist local $2$-forms $b_{i} \in \Omega ^{2}(U_{i})$ such that on overlaps 
$$\varphi _{ij}=exp (b_{j}-b_{i})$$
and the $3$-forms $db_{i}$ patch together to a global closed $3$-form $H$. Therefore $E$ is isomorphic, as a Courant algebroid, to the exact Courant algebroid $(TM\oplus T^{*}M,\hspace{0.2cm} [\cdot ,\cdot ]_{C}^{H},\hspace{0.2cm} \langle ,\rangle ,\hspace{0.2cm} \rho =\pi )$.
\item Finally, suppose that on each $E_{i}$ there is given a strong generalized almost complex structure $\mathfrak{J}_{i}$, Courant-integrable, and that these structures are compatible in the sense that:
$$\mathfrak{J}_{i}=\varphi _{ij}\circ \mathfrak{J}_{j}\circ \varphi _{ji}$$
on $U_{ij}$. Then the $\mathfrak{J}_{i}$ glue together to define a strong generalized almost complex structure on $E$, which is integrable with respect to the Courant bracket on $E$.
\end{nicenum}
\end{proposition}

\begin{proof}
Define
$$E \coloneqq \bigsqcup_i E_i \big/ \!\sim,$$
where the equivalence relation $\sim$ is given by
$$(x,u_i)\sim (x,u_j)\quad\Longleftrightarrow\quad x\in U_{ij}\ \text{and}\ u_i=\varphi_{ij}(u_j).$$
The compatibility hypothesis on the maps $\varphi_{ij}$ guarantees that $\sim$ is well defined. Hence we obtain a smooth projection
$$\pi:E\longrightarrow M,\qquad [(x,u)]\mapsto x.$$
Since each $E_{i}$ is a generalized tangent bundle, choose local trivializations
$$\Phi_i: E_i|_{U_i}\xrightarrow{\ \cong\ } U_i\times\mathbb R^{2n},\qquad n=\dim_{\mathbb R}M,$$
and let
$$\Psi_i:E|_{U_i}\longrightarrow U_i\times\mathbb R^{2n},\qquad \Psi_i([(x,u)])=\Phi_i(u).$$
Each $\Psi_i$ is a local trivialization and on overlaps one has
$$\Psi_i\circ\Psi_j^{-1}(x,v)=(x,g_{ij}(x)\,v),$$
where
$$g_{ij}(x):=\Phi_i|_{E_i{}_x}\circ\varphi_{ij}|_{E_j{}_x}\circ\Phi_j|_{E_j{}_x}^{-1}\in GL(2n,\mathbb R).$$
Since the $\varphi_{ij}$ are smooth in $x$, the transition maps $g_{ij}$ are smooth, satisfy the cocycle condition and are linear in the fibres. Therefore $E$ is a well-defined smooth vector bundle of rank $2n$ over $M$.
The elements of $E$ are equivalence classes $[(x,u)]$, and the space of smooth sections is
$$C^{\infty}(E)=\Big\{(s_1,\dots,s_m)\in\prod_{i=1}^m\Gamma(E_i)\ :\ s_i|_{U_{ij}}=\varphi_{ij}(s_j|_{U_{ij}})\ \text{for all }i,j\Big\}.$$
We now define a natural pairing on $E$ as follows. Let $e=[(x,u)]$, $f=[(x,v)]\in E_x$ and choose representatives $u_i,v_i\in (E_i)_x$ with $e=[(x,u_i)]$, $f=[(x,v_i)]$ for some $i$ with $x\in U_i$. Set
$$ \langle e,f\rangle_E(x)\coloneqq\langle u_i,v_i\rangle_{E_i}.$$
This is well defined because, on $U_{ij}$, using that each $\varphi _{ij}$ is a $b$-transform which preserves the pairing,
$$\langle u_j,v_j\rangle_{E_j}=\langle\varphi_{ji}u_i,\varphi_{ji}v_i\rangle_{E_j}=\langle u_i,v_i\rangle_{E_i},$$
and smoothness follows from the smoothness of the local pairings. \\ Define the anchor $\rho:E\to TM$ by
$$\rho([(x,u_i)])\coloneqq\rho_i(u_i),$$
where $\rho_i$ denotes the anchor of $E_i$. This is well defined since $\rho_j\circ\varphi_{ji}=\rho_i$ on overlaps. 
\\ For sections $s=(s_i)_i,t=(t_i)_i\in C^{\infty}(E)$ define the Courant bracket by
$$[s,t]^E_C\coloneqq\big([s_i,t_i]^{E_i}_C\big)_i.$$
This is well defined: on $U_{ij}$,
$$[s_i,t_i]^{E_i}_C=[\varphi_{ij}s_j,\varphi_{ij}t_j]^{E_i}_C=\varphi_{ij}\big([s_j,t_j]^{E_j}_C\big),$$
because each $\varphi_{ij}$ is an automorphism of the local Courant algebroid. Since each $E_i$ satisfies the Courant algebroid axioms and the $\varphi_{ij}$ preserve these structures, $E$ with the pairing $\langle\cdot,\cdot\rangle_E$, the anchor $\rho$ and bracket $[\cdot,\cdot]^E_C$ is a Courant algebroid.
\\ Moreover, the $\varphi_{ij}$ preserve the cotangent subbundle $T^*M$, hence $E$ is exact. Consequently there exist local $2$--forms $b_i\in\Omega^2(U_i)$ with $\varphi_{ij}=\exp(b_j-b_i)$ on $U_{ij}$, and the forms $db_i$ glue to a global closed $3$-form $H$, that defines the \v{S}evera class of $E$. Consequently $E$ is isomorphic (as a Courant algebroid) to $(TM\oplus T^*M,[\cdot,\cdot]^H_C,\langle\cdot,\cdot\rangle,\pi)$.
\\ Finally, suppose that each $E_i$ carries a strong generalized almost complex structure $\mathfrak J_i$ which is integrable with respect to the local Courant bracket, and that on overlaps
$$ \mathfrak J_i=\varphi_{ij}\circ\mathfrak J_j\circ\varphi_{ij}^{-1}\qquad\text{on }U_{ij}.$$
Choose local identifications
$$ \gamma_i:E|_{U_i}\xrightarrow{\ \cong\ }TU_i\oplus T^*U_i$$
such that $\gamma_i\circ\gamma_j^{-1}=\varphi_{ij}$ on $U_{ij}$. Define a global endomorphism $\mathfrak J:E\to E$ by
$$ \mathfrak J([(x,u)])\coloneqq \gamma_i^{-1}\big(\mathfrak J_i(\gamma_i([(x,u)]))\big)\qquad(x\in U_i).$$
We check that $\mathfrak J$ is well defined. If $x\in U_{ij}$ and $[(x,u)]=[(x,v)]$ then $v=\varphi_{ji}(u)$; hence
$$ \begin{aligned}
\mathfrak J([(x,u)]) &= \gamma_i^{-1}\big(\mathfrak J_i(\gamma_i([(x,u)]))\big)
= \gamma_i^{-1}\big(\mathfrak J_i(u)\big)\\
&= \gamma_i^{-1}\big(\varphi_{ij}\mathfrak J_j\varphi_{ij}^{-1}(u)\big)
= \gamma_j^{-1}\big(\mathfrak J_j(v)\big)
= \mathfrak J([(x,v)]).
\end{aligned}$$
Thus $\mathfrak J$ is well defined.
\\ Moreover, for every $[(x,u)]$ with $x\in U_i$,
$$\mathfrak J^2([(x,u)])=\gamma_i^{-1}\big(\mathfrak J_i^2(\gamma_i([(x,u)]))\big)
=\gamma_i^{-1}(-\mathrm{Id})(\gamma_i([(x,u)]))=-[(x,u)],$$
so $\mathfrak J^2=-\mathrm{Id}$. Orthogonality follows from the fact that each $\mathfrak J_i$ is orthogonal and $\gamma_i$ identifies the local pairings: for $e=[(x,u)]$, $f=[(x,v)]$ with $x\in U_i$,
$$ \langle\mathfrak J e,\mathfrak J f\rangle_E
= \langle\gamma_i(\mathfrak J e),\gamma_i(\mathfrak J f)\rangle
= \langle\mathfrak J_i u,\mathfrak J_i v\rangle
= \langle u,v\rangle
= \langle e,f\rangle_E.$$
Let $L_i\subset E_i\otimes\mathbb C$ be the $+i$-eigenbundle of $\mathfrak J_i$. The compatibility condition implies that the subbundles $\gamma_i^{-1}(L_i)\subset E|_{U_i}\otimes\mathbb C$ agree on overlaps, hence they glue to a global maximal isotropic subbundle $L\subset E\otimes\mathbb C$, which is the $+i$--eigenbundle of $\mathfrak J$.
\\ To prove integrability, let $s,t\in C^{\infty}(L)$. Locally $s$ and $t$ correspond to families $(s_i)_i,(t_i)_i$ with $s_i,t_i\in C^{\infty}(L_i)$ and $s_i=\varphi_{ij}(s_j)$ on overlaps. Since each $L_i$ is involutive with respect to the local Courant bracket, $[s_i,t_i]^{E_{i}}_{C}\in C^{\infty}(L_i)$ for every $i$. The global bracket was defined by $[s,t]^E_C=([s_i,t_i]^{E_{i}}_{C})_i$, and because the local brackets glue compatibly (the $\varphi_{ij}$ are Courant automorphisms), $[s,t]^{E}_{C}$ is a global section of $L$. Thus $L$ is involutive and $\mathfrak J$ is Courant-integrable.
\end{proof}

\begin{proposition}
Let $\mathfrak{J}_N = \begin{pmatrix} H_N & \alpha_N \\ \beta_N & K_N \end{pmatrix}$ be a strong generalized almost complex structure defined on the northern hemisphere $U_N \cong \mathbb{R}^6$ of $\mathbb{S}^6$. Suppose that:
\begin{enumerate}
    \item $\alpha_N$ satisfies at infinity: $\lim _{|x|\to \infty} \frac{\alpha _{N}(x)}{|x|^{4}}=0$.
    \item $H_N$ satisfies $\lim_{t\to\infty} H_N(tv) = H_\infty$ for every $v \in \mathbb{S}^5$.
\end{enumerate}
Then $\mathfrak{J}_N$ cannot be extended to a globally smooth generalized almost complex structure on $\mathbb{S}^6$ using a transition map composed of the stereographic projection differential and any smooth $b$-field transform where $b$ is continuously extended to $0$.
\end{proposition}

\begin{proof}
Let $\mathfrak{J}_N = \begin{pmatrix} H_N & \alpha_N \\ \beta_N & K_N \end{pmatrix}$ be a strong generalized almost complex structure on $U_{N}$.
Let $x$ and $y$ be the stereographic coordinates in the northern and southern chart, respectively. The transition map is given by $x = \phi(y) = \frac{y}{|y|^2}$, whose Jacobian matrix is $J_\phi(y) = \frac{1}{|y|^2} R(y)$, where $R(y) = I - 2\frac{yy^T}{|y|^2}$ is the orthogonal reflection matrix.
The transition map on $U_N \cap U_S$ is the composition of the pushforward $d\phi_{*}$ and a $b$-field transform $e^{b}$, so that  $\mathfrak{J}_S = e^b (d\phi_* \mathfrak{J}_N) e^{-b}$ on the southern chart.
Under the pushforward, the bivector block $\alpha _{N}$ is mapped to:
$$\tilde{\alpha}(y) = J_\phi(y)^{-1} \alpha_N\left(\frac{y}{|y|^2}\right) J_\phi(y)^{-T} = |y|^4 R(y) \alpha_N\left(\frac{y}{|y|^2}\right) R(y).$$
Since $b$-field transforms leave the bivector block invariant, the corresponding block of $\mathfrak{J}_{S}$ is $\alpha_S(y) = \tilde{\alpha}(y)$.
Smoothness of $\mathfrak{J}_S$ at the south pole requires the limit as $y \to 0$ to exist and be finite. Using condition 1, as $y \to 0$ (which corresponds to $|x| \to \infty$), the product $|y|^4 \alpha_N\left(\frac{y}{|y|^2}\right)$ vanishes. Since $R(y) \in O(6)$, we obtain:
$$\lim_{y \to 0} \alpha_S(y) = 0 \implies \alpha_S(0) = 0.$$
Since $\mathfrak{J}_S^2 = -Id$, evaluating the upper-left block of $\mathfrak{J}_S^2$ at $y=0$ yields:
$$H_S(0)^2 + \alpha_S(0)\beta_S(0) = -I \implies H_S(0)^2 = -I.$$
Furthermore, $H_{S}(y)$ is given by
$$H_S(y) = R(y) H_N\left(\frac{y}{|y|^2}\right) R(y) - \alpha_S(y) b(y).$$
Let $v\in \mathbb{S}^{5}$ and set $y=tv$. As $t \to 0$, we have $\alpha _{S}(tv)b(tv) \to 0$. By condition 2, $H_{N}(\frac{v}{t})\to H_{\infty}$. Thus we obtain:
$$H_S(0) = \lim_{t \to 0} \left[ R(v) H_N\left(\frac{v}{t}\right) R(v) \right] = R(v) H_\infty R(v).$$
Since this relation must hold for every unit vector $v \in \mathbb{S}^5$, we can isolate $H_\infty$ by multiplying on the left and right by $R(v)^{-1} = R(v)$, obtaining:
$$H_\infty = R(v) H_S(0) R(v) \quad \forall v \in \mathbb{S}^5.$$
Since the left-hand side is independent of $v$, for any two unit vectors $v, w \in \mathbb{S}^5$ we must have $R(v) H_S(0) R(v) = R(w) H_S(0) R(w)$, which rearranges to:
$$H_S(0) R(v) R(w) = R(v) R(w) H_S(0).$$
This shows that $H_S(0)$ commutes with all elements of the form $R(v) R(w)$. Since the products of two reflections generate the special orthogonal group $SO(6)$, $H_S(0)$ is an $SO(6)$-equivariant endomorphism. It is a well-known fact that the standard representation of $SO(6)$ on $\mathbb{R}^6$ is absolutely irreducible over $\mathbb{R}$. Therefore, by Schur's Lemma \cite{FH}, the only endomorphisms commuting with the entire group action are scalar multiples of the identity. 
Hence, we directly obtain $H_S(0) = \lambda I$ for some real scalar $\lambda \in \mathbb{R}$.
However, the condition evaluated earlier, $H_S(0)^2 = -I$, now yields:
$$(\lambda I)^2 = -I \implies \lambda^2 = -1.$$
This equation has no real solutions. This contradiction proves the claim.

\end{proof}

\begin{corollary} 
There exists no globally smooth generalized almost complex structure on $\mathbb{S}^{6}$ whose restriction to the northern stereographic chart $U_{N}\cong\mathbb{R}^{6}$ has constan components, and which is obtained on the southern chart through the transition map $e^{b}\circ d\phi_{*}$, with $b$ a $b$-field on $U_{N}\cap U_{S}$ extending continuously to the south pole.
\end{corollary}
\begin{proof}
Assume, for contradiction, that one can construct a global structure on $\mathbb{S}^6$ starting from a local model $\mathfrak{J}_N = \begin{pmatrix} H_N & \alpha_N \\ \beta_N & K_N \end{pmatrix}$ defined on the northern stereographic chart $U_N \cong \mathbb{R}^6$, where the tensor components are constant matrices. If $H_{N}$ and $\alpha_{N}$ are constant, then$\lim_{|x|\to\infty}\alpha_{N}(x)/|x|^{4}=0$ and $\lim_{t\to\infty}H_{N}(tv)=H_{N}$ for every $v\in\mathbb{S}^{5}$, so the hypotheses of Proposition 4.0.2 hold with $H_{\infty}=H_{N}$, and the claim follows. Explicitly, the proof of Proposition 4.0.2 gives $H_{S}(0)=\lambda I$ with $\lambda\in\mathbb{R}$, hence $H_{N}=R(v)H_{S}(0)R(v)=\lambda I$, while $\alpha_{S}(0)=0$ and $\mathfrak{J}_{S}^{2}=-\mathrm{Id}$ force $H_{S}(0)^{2}=-I$, or $\lambda^{2}=-1$, which is impossible.
\end{proof}


\begin{es}
As a comparison, consider the parallelizable manifold $\mathbb{T}^{6}$. Since $T\mathbb{T}^6$ is globally trivial, it admits an atlas where all geometric transition maps are the identity ($g_{ij} = I$).
If we equip each $U_{i}$ with the flat structure $\mathcal{J}_{-1,J_{0}}$ and take as transition maps b-field transforms $e^{b_{ij}}$ with $b_{ij}$ closed real $2$-forms of type $(1,1)$ with respect to $J_{0}$ (for instance $b_{ij}=0$), then $e^{b_{ij}}\mathcal{J}_{-1,J_{0}}e^{-b_{ij}}=\mathcal{J}_{-1,J_{0}}$, so the compatibility condition of Proposition 4.0.1 holds and these local structures glue to a globally well-defined, Courant-integrable strong generalized complex structure on $T^{6}$. This shows that the obstruction of Proposition 4.0.2 is not an artifact of the $b$-field gluing procedure itself but it arises from the transition map $d\phi$ between the stereographic charts of $\mathbb{S}^{6}$, which has no analogue on $T^{6}$.
\end{es}

\nocite{*}
\bibliographystyle{plain}
\bibliography{bibliografia}

\end{document}